\documentclass[reqno]{amsart}
\usepackage{amsmath,amsthm,amsfonts,amssymb}

\usepackage{hyperref}
\usepackage[all]{xy}
\usepackage{graphicx,import}
\usepackage{amscd}
 \usepackage{color}
\usepackage{enumerate}
\usepackage{fourier}
\usepackage{tikz}
\usepackage{cases}

\hypersetup{
    colorlinks=true,
    linkcolor=blue,
    citecolor=green,
    urlcolor=cyan
}

\numberwithin{equation}{section}

\theoremstyle{plain} 
\newtheorem{theorem}{Theorem}[section]
\newtheorem{corollary}[theorem]{Corollary}
\newtheorem{lemma}[theorem]{Lemma}
\newtheorem{prop}[theorem]{Proposition}

\theoremstyle{definition}
\newtheorem{definition}[theorem]{Definition}

\newtheorem{example}[theorem]{Example}
\newtheorem{conjecture}[theorem]{Conjecture}

\newtheorem{notation}[theorem]{Notation}
\newtheorem{assumption}[theorem]{Assumption}

\theoremstyle{remark}
\newtheorem{remark}[theorem]{Remark}

\newcommand{\PP}{\mathbb{P}}

\newcommand{\Z}{\mathbb{Z}}
\newcommand{\R}{\mathbb{R}}
\newcommand{\C}{\mathbb{C}}

\newcommand{\bL}{\mathbb{L}}

\newcommand{\CA}{\mathcal{A}}
\newcommand{\CB}{\mathcal{B}}

\newcommand{\CC}{\mathcal{C}}

\newcommand{\CF}{\mathcal{F}}

\newcommand{\ks}{\mathfrak{ks}}

\newcommand{\OL}[1]{\overline{#1}}

\newcommand{\AI}{{A_\infty}}

\newcommand{\Jac}{{\rm Jac}}

\newcommand{\chicheck}{\check{\chi}}
\newcommand{\cchi}{\check{\chi}}

\newcommand{\HG}{{\widehat{G}}}
\newcommand{\id}{\mathrm{id}}
\newcommand{\wG}{\widehat{G}}

\newcommand{\WT}[1]{\widetilde{#1}}
\newcommand{\WH}[1]{\widehat{#1}}

\newcommand{\Hom}{{\rm Hom}}
\newcommand{\be}{\mathbf{1}}

\allowdisplaybreaks

\begin{document}
\title[Kodaira-Spencer map, Lagrangian Floer theory and orbifold Jacobian algebras]{Kodaira-Spencer map, Lagrangian Floer theory \\ and orbifold Jacobian algebras}
\author{Cheol-Hyun Cho}
\author{Sangwook Lee}
\address{Cheol-Hyun Cho: Department of Mathematical Sciences, Research Institute in Mathematics, Seoul National University, Gwanak-gu, Seoul, South Korea}
\email{chocheol@snu.ac.kr}
\address{Sangwook Lee: Korea Institute for Advanced Study\\ 85 Hoegi-ro Dongdaemun-gu, Seoul 02455, South Korea}
\email{swlee@kias.re.kr}
\begin{abstract}
A version of mirror symmetry predicts a ring isomorphism between quantum cohomology of a symplectic manifold and Jacobian algebra of the Landau-Ginzburg mirror, and for toric manifolds Fukaya-Oh-Ohta-Ono constructed such a map called Kodaira-Spencer map using Lagrangian Floer theory.
We discuss a general construction of Kodaira-Spencer ring homomorphism when LG mirror potential $W$ is given by $J$-holomorphic discs with boundary on a Lagrangian $L$:
We find an $\AI$-algebra $\mathcal{B}$ whose $m_1$-complex is a Koszul complex for $W$ under mild assumptions on $L$.
Closed-open map gives a ring homomorphism from quantum cohomology to cohomology algebra of $\CB$ which is Jacobian algebra of $W$.

We also construct an equivariant version for orbifold LG mirror $(W,H)$.  We construct a Kodaira-Spencer map
from quantum cohomology to another $\AI$-algebra $(\CB\rtimes H)^H$ whose cohomology algebra is isomorphic to the orbifold Jacobian algebra of $(W,H)$ under an assumption. 
For the $2$-torus whose mirror is an orbifold LG model given by Fermat cubic with a $\Z/3$-action,
we compute an explicit Kodaira-Spencer isomorphism.
\end{abstract}
\maketitle

\section{Introduction}
Let $X$ be a symplectic manifold with a possibly immersed Lagrangian submanifold $\bL$ on it.
Fukaya, Oh, Ohta and Ono defined an $\AI$-algebra for $\bL$ in \cite{FOOO} (Akaho-Joyce \cite{AJ} for immersed case) as well as a potential function $W_\bL:Y \to \Lambda$ on its deformation space $Y$ using $J$-holomorphic discs. 
Suppose  $\bL$ is essential in $X$ in the sense that it generates Fukaya category of $X$, or
more informally  $\bL$ is a representing Lagrangian object for the skeleton of $X \setminus D$ for
a choice of  anti-canonical divisor $D$ of $X$. Then $W_\bL$ may be taken as a 
mirror Landau-Ginzburg model for $X$(\cite{A}).  Main examples are Lagrangian torus fibers in  toric manifolds \cite{CO}, \cite{FOOO}, and Seidel Lagrangian in orbi-sphere $\mathbb{P}^1_{a,b,c}$ \cite{S},\cite{E}, \cite{CHL}. Higher dimensional analogue of the latter example is an immersed sphere constructed by Sheridan in higher dimensional pair of pants \cite{Sh} or in  Fermat type hypersurfaces \cite{Sh2} although weakly unobstructedness of $\bL$ in these cases are not yet proved in general. 

Once $W_\bL$ is constructed from Maurer-Cartan equation of $\bL \subset X$,  the first author with Hong and Lau constructed a canonical $\AI$-functor from Fukaya category of $X$ to 
the matrix factorization $\AI$-category of $W$, called a localized mirror functor \cite{CHL,CHL2}. This functor is constructed as a curved version of Yoneda embedding relative to a fixed Lagrangian $\bL$, called reference Lagrangian. Hence the construction of the functor is based on Fukaya category operations and provides a  geometric way to understand homological mirror symmetry between these mirror pairs at hand. 

Let us turn our attention to closed string mirror symmetry.  Historically, closed string mirror symmetry between
quantum cohomology and Jacobian algebra of $W$ (or their Frobienius manifolds) has been proved  much earlier (Batyrev \cite{B}, 
Givental \cite{Gi}, Iritani \cite{Ir} and so on) than homological mirror symmetry. These approaches are based on study of Gromov-Witten invariants and Picard-Fuchs equations. 

Fukaya, Oh, Ohta and Ono \cite{FOOO_MS} introduced a Lagrangian Floer theoretic approach to study closed string mirror symmetry.
Namely, they constructed a geometric map, called Kodaira-Spencer map from quantum cohomology to Jacobian algebra of $W$
in the case of toric manifolds, and $T^n$-action plays an essential role there. 
Here Jacobian algebra of $W$ (which is also called Milnor ring) is a polynomial algebra modulo the ideal generated by the partial derivatives of $W_\bL$.
 Recently,  Amorim, Hong, Lau and the first author \cite{ACHL} generalized their construction to the case of $\mathbb{P}^1_{a,b,c}$ orbifold by using  $\Z/2$-action instead of $T^n$-action.  
 
 In this paper, we define Kodaira-Spencer ring homomorphism in much more generality, including the case of any weakly unobstructed Lagrangian torus $\bL$ in
 a symplectic manifold $M$ with Lagrangian Floer potential $W_\bL$.  The main idea is to replace Jacobian ring  with the Koszul complex of $W_\bL$, which
can be geometrically constructed from  Maurer-Cartan theory of the Lagrangian submanifold. 

Kodaira-Spencer map  provides a nice geometric explanation for the isomorphism between complicated quantum multiplication and a trivial multiplication in Jacobian algebra, which can be easily extended for bulk deformations. Also, as shown in \cite{FOOO_MS}, this provides a natural setting to compare Poincar\'e duality pairing of symplectic manifolds and residue pairing of $W_\bL$ in complex geometry. More recently, we have shown in \cite{CLS} that these two pairings are conformally equivalent with conformal factor given by the ratio of quantum volume form and classical volume form of $\bL$.  The results of \cite{CLS} was restricted to the case of
toric manifolds and $\mathbb{P}^1_{a,b,c}$ due to the absence of general Kodaira-Spencer map.
Therefore, the construction in this paper can be used to prove a generalization of the results in \cite{CLS} for the cases of $(M,\bL)$ satisfying Cardy relations.

Often mirror symmetry involves a finite group symmetry. In the very first example of Fermat quintic 3-folds \cite{COGP}
as well as in more recent works of Seidel \cite{S1}, \cite{S} and Sheridan \cite{Sh} such symmetry plays an important role for the proof of HMS conjecture.
In this paper, we will consider a closed string mirror symmetry between $QH^*(X)$ and orbifold Jacobian algebra $\Jac(W,\WH{G})$.
As far as the authors know, there has been no known examples, partly because the product structure of $\Jac(W,\WH{G})$ is
just beginning to be understood.   Let us suppose that $\WH{G}$ is a finite abelian group acting linearly on variables and $W$ is a $\WH{G}$-invariant polynomial. 
Orbifold Jacobian algebra $\Jac(W,\WH{G})$ is easy to describe as a module, which is given by the Jacobian algebra
of $W$ restricted to the fixed part of $\chi$-action for each $\chi \in \WH{G}$ (see Definition \ref{def:ojr2}).  But it has quite an interesting and mysterious product structure, which has been the focus of several interesting recent works (see \cite{BT}, \cite{BTW}, \cite{HLL}, \cite{Shk}).

In this paper, we  will define an equivariant version of  Kodaira-Spencer map, namely a ring homomorphism from quantum cohomology to the orbifold Jacobian algebra
under some assumption on the Lagrangian. We will describe an explicit equivariant Kodaira-Spencer isomorphism for $T^2$ at the end.

Let us give more detailed summary of our construction and precise theorems.
We first recall the idea behind the construction of Kodaira-Spencer map by Fukaya-Oh-Ohta-Ono.
Recall that closed-open map in $A$-model usually defines a map from quantum cohomology of $X$ to a Hochschild cohomology of Fukaya category of $X$.
Fukaya-Oh-Ohta-Ono observed that instead of the full Hochschild cochains, one may focus on the cochains with no input and one output,
coming from holomorphic discs with boundary on a Lagrangian torus fiber $\bL$.
Since  Lagrangian $\bL$ is given by an orbit of $T^n$-action, there is a free $T^n$-action on the moduli space of holomorphic discs with  at least one boundary marked point. 
Given $\alpha \in QH^*(X)$,  the (virtual) evaluation image on $\bL$ of $J$-holomorphic discs which intersect $\alpha \in QH^*(X)$ in the interior
is always a multiple of fundamental class $[\bL]$ (from $T^n$-action) and $\mathfrak{ks}:QH^*(X) \to \Jac(W_\bL)$ is defined by
reading the coefficient of $[\bL]$. 
But unfortunately simple dimension counting  of closed-open map shows that  the image of a closed-open map is not expected to become a multiple of $[\bL]$ in general and this was the main obstacle to define Kodaira-Spencer map beyond the above examples.

A key idea is to replace Jacobian algebra by an associated Koszul complex and to observe that MC formalism of $\bL$
provides such a complex under a mild assumption on Lagrangian $\bL$. More precisely, given an $\AI$-algebra $\mathcal{A}$ of $\bL$, we  define a new $\AI$-algebra $\mathcal{B}$
by tensoring $\Lambda[x_1,\cdots,x_n]$ and by considering an $\AI$-operation $m_k^b$ for the family of weak MC elements $b = \sum x_i X_i$.
The cohomology of $\mathcal{B}$ is isomorphic to Jacobian ring of $W_\bL$ under a simple assumption \ref{as1},
which states that for $\AI$-algebra $\mathcal{A}$, degree one cohomology classes $X_1,\cdots,X_n$ generate the cohomology algebra $H^*(\mathcal{A})$ of rank $2^n$. This assumption is satisfied for any Lagrangian torus in a symplectic manifold or Seidel Lagrangian in $\mathbb{P}^1_{a,b,c}$. 

With this setup, we  use the closed-open map with boundary deformation by MC elements $b$ to define a ring homomorphism 
$$\mathfrak{ks}: QH^*(X) \to H^*(\mathcal{B})_{alg}.$$
Here we denote by $H^*(\mathcal{B})_{alg}$  the cohomology algebra of the complex $(\mathcal{B},m_1^b)$ with an associative
product $v \cdot w := (-1)^{|v|} m_2^b(v,w)$. 

Our setup enables us to apply the construction of Fukaya, Oh, Ohta and Ono \cite{FOOO_MS} ( small modification of \cite{ACHL})
to define a general (localized) Kodaira-Spencer map. Note that we do not require any symmetry on $\bL$ other than its weakly unobstructedness.
\begin{theorem}[Theorem \ref{thm:ks}]\label{thm:1}
Let $L$ be an object of Fukaya category of $M$ and let $\mathcal{A}(L)$ be its Fukaya algebra, with bounding cochain $b$
and potential function $W$. Denote by $\mathcal{B}(L)$ the $\AI$-algebra constructed in Definition \ref{def:B}.
There is a  Kodaira-Spencer map which is a ring homomorphism
$$\mathfrak{ks}:QH^*(M,\Lambda) \to H^*(\mathcal{B}(L))_{alg}.$$
Under the Assumption \ref{as1}, $H^*(\mathcal{B}(L))_{alg}$ is isomorphic to $\Jac(W_L)$ by Proposition \ref{prop:bjac}.
\end{theorem}
We remark that  further geometric investigation is needed to show that $\ks$ is an isomorphism.
In the case of \cite{FOOO}, the images of toric divisors under $\ks$ is analyzed to prove an isomorphism property.
In \cite{ACHL}, authors had to enlarge the domain of Landau-Ginzburg mirror (valuations of variables in a Novikov field) to make $\ks$ an isomorphism.

Let us move on to the equivariant cases.
Let $X$ be a symplectic manifold with a finite symmetry group $G$. Then the quotient $[X/G]$ is a symplectic orbifold.
Suppose we have a possibly immersed Lagrangian submanifold $\OL{L}$ on $[X/G]$ such that it has embedded 
Lagrangian lifts $\WT{L}$ in $X$. By applying the previous construction on $\OL{L} \subset [X/G]$, we obtain a Fukaya $\AI$-algebra for $\OL{L}$ with a potential function  $W_{\OL{L}}:Y \to \Lambda$.
  Dual group $\WH{G}$ naturally acts on $Y$ leaving $W$ invariant and we obtain the mirror Landau-Ginzburg orbifold $(W_{\OL{L}},\WH{G})$. Here the action of $\chi \in \WH{G}$ on the mirror variable $x_i$  is determined by the change
  of branchs of $\WT{L}$ for the corresponding Lagrangian intersection $X_i$(Definition \ref{def:hatac}).
  
Now, suppose $\bar{L}$ is essential in $[X/G]$ and then the mirror of the orbifold $[X/G]$ should be the LG model $W := W_{\OL{L}}$.
Also the mirror of $X$ should be  the orbifold LG model $(W, \WH{G})$.
In this setup, the first author together with Hong and Lau constructed homological mirror symmetry functor (\cite{CHL},\cite{CHLnc}).
Namely, there exist a localized mirror functor $\mathcal{F}^{\bar{L}}:  \mathrm{Fu}([X/G]) \to MF^\AI(W)$.
Here we define $\mathrm{Fu}([X/G])$ to be  $G$-invariant part $\mathrm{Fu}(X)^G$. 
There is also an upstairs localized mirror functor 
$$\mathcal{F}^{\WT{L}}:  \mathrm{Fu}(X) \to MF_{\WH{G}}^\AI(W)$$
to the $\WH{G}$-equivariant matrix factorization category of $W$ (see \cite{CHLnc} for more details).

In this paper, we find a direct connection of orbifold Jacobian algebra to symplectic geometry.
Namely, we  define a  Kodaira-Spencer map from quantum cohomology to the orbifold Jacobian algebra,
which is an equivariant version of Theorem \ref{thm:1}. 
We will define an $\AI$-algebra $\mathcal{B}(\OL{L}) \rtimes \WH{G}$ and an additional $\WH{G}$-action on it
and define an equivariant version of Kodaira-Spencer map.
\begin{theorem}[Theorem \ref{thm:ks2}]
Suppose the Lagrangian $\OL{L} \subset [X/G]$ and its potential function $W$ is given as above.
There is a Kodaira-Spencer map which is a ring homomorphism
\begin{equation}\label{eq:ksb}
\mathfrak{ks}:QH^*(M,\Lambda) \to H^*\big(( \mathcal{B}(\OL{L}) \rtimes \WH{G})^{\WH{G}}\big)_{alg}.
\end{equation}
\end{theorem}
We prove that under suitable assumption,  $H^*\big(( \mathcal{B}(\OL{L}) \rtimes \WH{G})^{\WH{G}}\big)_{alg}$ is isomorphic to orbifold Jacobian algebra.
\begin{theorem}[Theorem \ref{thm:main}]
We have
$$H^*\big(( \mathcal{B}(\OL{L}) \rtimes \WH{G})^{\WH{G}}\big)_{alg} \cong \Jac(W, \WH{G}).$$
under the following  Assumption \ref{as2} on a reference Lagrangian $L$ (which is $\OL{L}$ in this case)with potential function $W(x)$:
\begin{equation}\label{eq:cfmf}
\big(CF(L,L)\otimes \Lambda[x_1,\cdots,x_n,y_1,\cdots,y_n],  m_1^{b(x),b(y)} \big)
\end{equation}
defines a matrix factorization of $W(y)-W(x)$ and it is isomorphic to the Koszul matrix factorization for
the regular sequence $(y_1-x_1, \cdots, y_n-x_n)$.
\end{theorem}
Let us first explain the Assumption \ref{as2}.
Conjecturally, the  Assumption \ref{as2} for an essential Lagrangian $L \subset M$  may be formulated as a correspondence between geometric diagonal  $\Delta \subset M^{-} \times M$
and an algebraic diagonal of $W(y) -W(x)$ which is the kernel of identity functor of matrix factorization category $MF(W)$. Therefore, we expect this assumption to hold in such cases. We hope to discuss this in more detail elsewhere.
In this paper, we show that monotone Lagrangian tori in a symplectic manifold as well as Seidel Lagrangian in $\mathbb{P}^1_{a,b,c}$ satisfies this assumption  \ref{as2} more directly. 

Let us also briefly explain the idea of the proof of the above theorem  \ref{thm:main}.
The semi-direct product $\AI$-algebra $( \mathcal{B}(\OL{L}) \rtimes \WH{G})$ concerns Maurer-Cartan data of the Lagrangian lifts $\WT{L}$,
while keeping the same variables $(x_1,\cdots,x_n)$. We apply localized mirror functor relative to the reference  $(\OL{L},b(y))$ to
obtain the matrix factorization \eqref{eq:cfmf} of $W(y) - W(x)$. The rest of the construction is to consider the additional $\WH{G}$-action
to compare it with the equivariant Kernel $\Delta_W^{\WH{G}\times \WH{G}}$ for matrix factorization category.

We give an explicit example. It turns out that $\mathfrak{ks}$ for a symplectic torus $T^2$ is already non-trivial. Recall that Polishchuk-Zaslow \cite{PZ} explained the HMS of $T^2$ from the point of view of Strominger-Yau-Zaslow formalism\cite{SYZ} that mirror pairs are obtained as dual torus fibrations. In \cite{CHL} (following the work of Seidel \cite{S}),  the $\Z/3$-quotient of $T^2$ and a Seidel Lagrangian $\OL{L} \in [T^2/\Z/3]$
was considered. In \cite{CHL}, $\OL{L}$ is shown to be weakly unobstructed with a potential function
\[W=-\phi(q) i(x_1^3+x_2^3+x_3^3)+\psi(q) i x_1x_2x_3.\]
Here $\phi, \psi$ are power series in $q$ given in \eqref{phipsi}. Localized mirror functor provides an explicit HMS equivalence in this case.

We prove the following theorem, by computing the Kodaira-Spencer map and comparing
with the algebraic generators of orbifold Jacobian algebra by Shklyarov \cite{Shk}.
\begin{theorem}\label{KS} 
A Kodaira-Spencer ring isomorphism $\ks: QH^*(T^2,\Lambda) \to \Jac(W,\Z/3)$ is given by 
\begin{eqnarray*}
\mathrm{pt}_{T^2} & \mapsto & \frac{1}{24}q\frac{\partial W}{\partial q}, \\
l_\chi:=-i{\gamma} \cdot \frac{[C_h]-\chicheck [C_v]}{\chicheck-1} & \mapsto & \xi_\chi,\\
l_{\chi^2}:=-i{\gamma} \cdot \frac{[C_v]-\chicheck [C_h]}{\chicheck-1} & \mapsto & \xi_{\chi^2}.\\
\end{eqnarray*}
Here $\gamma$ is a modular form 
\[ \gamma=\sum_{k \in \Z} (-1)^k i q^{(6k+1)^2}.\]
\end{theorem}

The organization of this paper is as follows.
We first review the notion of a kernel for a functor between categories of matrix factorizations in Section \ref{sec:ker}.
The kernel $\Delta_W$ for the identity functor $\id:MF(W) \to MF(W)$ is a Koszul matrix factorization and
$\Hom(\Delta_W,\Delta_W)$ describes the Hochschild cohomology of $MF(W)$, or the Jacobian ring of $W$.
We review the algebraic definition of orbifold Jacobian algebras in Section \ref{sec:oja} following Shklyarov \cite{Shk}.
In section \ref{sec:mc}, we explain how to use Maurer-Cartan formalism of $\AI$-algebra  $\mathcal{A}$ by Fukaya-Oh-Ohta-Ono 
to define a new $\AI$-algebra $\mathcal{B}$. We describe Assumption \ref{as1} which implies that the cohomology algebra of $\mathcal{B}$ is isomorphic to Jacobian algebra of the superpotential $W$ of $\mathcal{A}$.
And we define Floer theoretic kernel $\Delta_{Fl}$ using assumption \ref{as2}.
In section \ref{sec:le}, we show that monotone Lagrangian torus as well as Seidel Lagrangian in $\mathbb{P}^1_{a,b,c}$ satisfies this
assumption \ref{as2} (which implies \ref{as1}).
In section \ref{sec:eqb}, we give an equivariant analogue of the construction in  section \ref{sec:mc}. 
From an $\AI$-algebra $\WT{\mathcal{A}}$ with finite group $G$-action, we define
a new $\AI$-algebra $\mathcal{B} \rtimes \WH{G}$, by studying equivariant version of Maurer-Cartan formalism.
In section \ref{sec:eqbj}, we show that $H^*\big((\mathcal{B} \rtimes \WH{G})^{\WH{G}}\big)_{alg}$
is quasi-isomorphic to orbifold Jacobian algebra under the assumption \ref{as2}. For the proof, we use
the notion of equivariant Kernel $\Delta_W^{\wG\times \wG}$ and show that  the $\AI$-algebra $(\mathcal{B} \rtimes \WH{G})^{\WH{G}}$ is shown to be $\AI$-quasi-isomorphic to $hom_{MF^\AI(W(y)-W(x))}(\Delta_W^{\wG\times \wG}, \Delta_W^{\wG\times \wG})$ whose
cohomology algebra (in dg sign convention) is isomorphic to the orbifold Jacobian algebra.
In section \ref{sec:KS}, we construct Kodaira-Spencer map from quantum cohomology to  cohomology algebra of $\mathcal{B}$
or $(\mathcal{B} \rtimes \WH{G})^{\WH{G}}$. In section \ref{sec:ag}, we review the construction of Shklyarov so that
we can match our construction to that of \cite{Shk}. In section \ref{sec:t2}, we illustrate our construction to compute the Kodaira-Spencer map for $T^2$ using geometric construction. In Appendix \ref{algjac}, we write the algebraic computation of
the orbifold Jacobian algebra for the mirror of $T^2$.

\begin{notation}
We introduce a few notations. Here $k$ is a base field.
For mirror symmetry $k$ is  Novikov field $\Lambda$ over $\C$ which is algebraically closed with a valuation $\nu:\Lambda \to \R$.
\begin{flalign*}
&R:=k[x_1,\cdots,x_n],\\
&R^e:= k[x_1,\cdots,x_n,y_1,\cdots,y_n],\\
&\Lambda_0:=\Big\{ \sum_{j=1}^\infty a_j T^{\lambda_j} \mid  a_j \in \mathbb{C},\; \lambda_j\in \R_{\geq 0}, \; \lim_{j\to \infty} \lambda_j=\infty \Big\}, \\
&\Lambda:= \Lambda_0[T^{-1}], \;\; \Lambda_+:=\Big\{ \sum_{j=1}^\infty a_j T^{\lambda_j} \in \Lambda_0 \mid \lambda_j>0 \Big\}, \\
&\nu(\sum_{j=1}^\infty a_j T^{\lambda_j}) := \min_j (\lambda_j).
\end{flalign*}

Here $R$ and $R^e$ are suitably completed when $k = \Lambda$. 
For example $R$ (for $k=\Lambda$) is the convergent power series ring  which is the set of elements 
$$\sum_{i_1,\cdots, i_n \geq 0} c_{i_1,\cdots,i_n} x_1^{i_1} \cdots x_n^{i_n}$$ such that
the elements $c_{i_1,\cdots,i_n} \in \Lambda$ satisfies $ \lim_{i_1+\cdots + i_n \to \infty} \nu(c_{i_1,\cdots,i_n}) = \infty$.
\end{notation}

\subsection*{Acknowledgements} We would like to thank Lino Amorim, Yong-Geun Oh, David Favero and Yoosik Kim for very helpful discussions.
C.-H. Cho was supported by the NRF grant funded by the Korea government(MSIT) (No. 2017R1A22B4009488).
S. Lee was supported by an Individual Grant(MG063902) at Korea Institute for Advanced Study and by the Fields Institute during the Thematic Program on Homological Algebra of Mirror Symmetry.

\section{Kernels for matrix factorizations}\label{sec:ker}
Let us recall our basic set-up for matrix factorization category.
Let $W$ be an element in $R=k[x_1,\cdots,x_n]$ for an algebraic closed field $k$ of characteristic 0. We assume that 
$W $ has only isolated singularity at the origin. 
\begin{definition}
A {\em matrix factorization} of $W$ is a $\Z/2$-graded projective $R$-module $P=P_0\oplus P_1$ together with a morphism $d=(d_0,d_1)$ of degree 1 such that
\[ d^2=W\cdot \id.\] 
A {\em morphism} $\phi: (P,d) \to (Q,d')$ of {\em degree} $j\in \Z/2$ is a map of $\Z/2$-graded $R$-modules of degree $j$.  Define
\[ D\phi:=d'\circ \phi -(-1)^{|\phi|}\phi \circ d\]
and define the composition of morphisms in the usual sense. This defines a dg-category of matrix factorizations $(MF(W),D,\circ)$.
\end{definition}

This dg category can be turned into an $\AI$-category $MF^\AI(W)$ as follows:
\begin{definition}\label{def:mfa}
The $\AI$-category  $MF^\AI(W)$ has the same set of objects as  $(MF(W),D,\circ)$.
\[ hom_{MF^\AI(W)}(E,F):=hom_{MF(W)}(F,E),\]
\begin{equation}\label{eq:dgainfty} m_1(\Phi):= D\Phi,\; m_2(\Phi,\Psi):= (-1)^{|\Phi|}\Phi\circ \Psi,\end{equation}
and $ m_{k} =0$ for any other $k$.
\end{definition}
We still denote a morphism $\Phi \in hom_{MF^\AI(W)}(E,F)$ as $\Phi:F \to E$.
\begin{remark}\label{rem:dgmodification}
Conversely, given an $\AI$-algebra $\mathcal{C}$ with $m_0=0$, its $m_1$-cohomology
$H^*(\mathcal{C})$ has an induced associative algebra structure $\circ$ given by
\[ \Phi \circ \Psi:= (-1)^{|\Phi|}m_2(\Phi,\Psi).\]
To emphasize that we need this additional sign, we will denote such cohomology algebra as $H^*(\mathcal{C})_{alg}$.
\end{remark}

There is a class of matrix factorizations, called {\em Koszul matrix factorizations}, which play a  central role in the theory.
Let $I$ be an ideal of $R$ generated by a regular sequence $(b_1,\cdots, b_n)$, and suppose that $W \in I$.
Let $(a_1,\cdots,a_n)$ be elements of $R$ such that 
$$W = \sum_{i=1}^n a_i \cdot b_i.$$

\begin{definition} The {\em $n$-th graded Clifford algebra} is
\[ \mathrm{Cl}_n:=k[\theta_1,\cdots,\theta_n,\partial_{\theta_1},\cdots,\partial_{\theta_n}],\;\; |\theta_i|=-1,\;\; |\partial_{\theta_i}|=1 \;\;\text{for all $i$}\]
together with relations
\begin{equation}\label{eq:cliffrel}
 \theta_i \theta_j=-\theta_j \theta_i,\;\; \partial_{\theta_i}\partial_{\theta_j}=-\partial_{\theta_j}\partial_{\theta_i},\;\; \partial_{\theta_i}\theta_j=-\theta_j\partial_{\theta_i}+\delta_{ij}.
 \end{equation}
\end{definition}

From the regular sequence $\{\vec{b}\}$, the associated Koszul complex is defined as follows.
\begin{equation}\label{eqK}
K^\bullet = \big( k[\theta_1,\cdots, \theta_n] \otimes R,s_0 \big),\;\; \textrm{with} \;\; s_0 = \sum_j b_j \partial_{\theta_j}.
\end{equation}

The Koszul matrix factorization for $\{\vec{a},\vec{b}\}$ is defined as $(K^\bullet ,\delta \big)$, 
where $$\delta =  s_0+ s_1, \;\; \textrm{with} \;\; s_1(\alpha)= \sum_{j=1}^n a_j \theta_j \cdot \alpha .$$
 It is a result of Eisenbud that this Koszul matrix factorization corresponds to the stabilization of
$R/W$-module $R/I$ (see Proposition 2.3.1 of \cite{PV} for example). Note that corresponding module $R/I$ does not depend on the choice of a regular sequence $\vec{b}$ generating $I$ nor $\vec{a}$.

We recall a few results using derived Morita theory of matrix factorizations.
\begin{definition}[\cite{Toen}]\label{def:hhcohomology}
Let $\CA$ be a dg category. The {\em Hochschild cochain complex} of $\CA$ is defined as
\[ C^*(\CA,\CA):= hom_{R\underline{Hom}_c(\CA,\CA)}(\mathrm{Id}_\CA,\mathrm{Id}_\CA)\]
where $R\underline{Hom}_c(\CA,\CA)$ is the dg category of continuous functors from $\CA$ to itself.
\end{definition}

By modifying results in \cite{Toen} for differential $\Z/2$-graded categories and applying them to matrix factorizations, Dyckerhoff showed the following.
\begin{theorem}[\cite{Dyc}]
There is a quasi-equivalence between two dg categories
\begin{equation}\label{eq:ker}
 R\underline{Hom}_c(MF(R,W),MF(R',W')) \simeq MF(R\otimes_k R', -W\otimes 1+ 1\otimes W').
 \end{equation}
In the case that  $(R',W')=(R,W)$, identity functor $\mathrm{Id}_{MF(R,W)}$ corresponds via the above quasi-equivalence to a matrix factorization, that is given by a resolution of 
\[R\cong R^e/(y_1-x_1,\cdots,y_n-x_n)\] 
as an $R^e/(W(y_1,\cdots,y_n)-W(x_1,\cdots,x_n))$-module.
\end{theorem}
\begin{definition}
For $\CF \in  R\underline{Hom}_c(MF(R,W),MF(R',W'))$, 
 the corresponding matrix factorization under the quasi-equivalence \eqref{eq:ker} is called a {\em kernel} for $\CF$.
\end{definition}
We remark that  a kernel for identity functor is not unique.
We take the following Koszul matrix factorization and denote it by $\Delta_W$.
For $j=1,\cdots,n$, define polynomials $\Delta_jW \in R^e$ (also written as $\nabla^{x\to (x,y)}_j(W)$ later)  by
\[ \Delta_j W:= \frac{W(x_1,\cdots,x_{j-1},y_j,\cdots,y_n)-W(x_1,\cdots,x_j,y_{j+1},\cdots,y_n)}{y_j-x_j}.\]
Then we have
\begin{equation}\label{eq:ksmf}
W(y_1,\cdots,y_n) - W(x_1,\cdots,x_n) = \sum_{j=1}^n \Delta_jW \cdot (y_j- x_j),
\end{equation}
From now on, we write $W(x)=W(x_1,\cdots,x_n),\;W(y) = W(y_1,\cdots,y_n)$ for simplicity.
\begin{corollary} We have
\[ C^*(MF(R,W),MF(R,W))\simeq hom_{MF(R^e,W(y) - W(x))}(\Delta_W,\Delta_W).\]
\end{corollary}

Now, let us recall the notion of equivariant matrix factorizations. Let $H$ be a finite group, and suppose $H$ acts on $R$ and that $W$ is invariant under this action.

\begin{definition}
An {\em $H$-equivariant matrix factorization} of $W$ is a matrix factorization $(P,d)$ where $P$ is equipped with an $H$-action, and $d$ is $H$-equivariant. An {\em $H$-equivariant morphism} $\phi$ between two $H$-equivariant matrix factorizations $(P,d)$ and $(Q,d')$ is an $H$-equivariant morphism of $\Z/2$-graded $R$-modules $P$ and $Q$.
\end{definition}
It is well-known that $H$-equivariant matrix factorizations with $H$-equivariant morphisms also form a differential $\Z/2$-graded  category $MF_H(W)$.
\begin{example}
Let $R = \C[z]$ with the polynomial $W=z^n$.
Define a $\Z/n$-action by setting $[1] \cdot z = \xi z$ for an $n$th root of unity $\xi$ and $[1] \in \Z/n$, which
preserves $W$. For any $k$ (where  $0 \leq k \leq n$), we can define a matrix factorization given by $P_0 = R\langle e\rangle$ and $P_1 = R\langle o\rangle$
with   $d_0(e) = z^k o, d_1(o) = z^{n-k} e$.
For any $ 0 \leq l < n$, we can set 
\[[1]\cdot e = \xi^{k+l} e, \; [1]\cdot o= \xi^l o\]
to give a $\Z/n$-equivariant structure to this matrix factorization $(P_\bullet, d_\bullet)$.
\end{example}
Let us explain the equivariant structure of $\Delta_W$ following Section 2.5 of \cite{PV}.
Suppose a finite group $H$ acts on $R=k[x_1,\cdots,x_n]$, hence $H$ also acts diagonally on $R^e$. $\Delta_W$ can be modified to produce a $H$-equivariant matrix factorization $\Delta_W^{H}$ as follows. Define $H$-action on $\theta_1,\cdots,\theta_n$ to be the same as that on $x_1,\cdots,x_n$ which induces  $H$-action on $\{\partial_{\theta_i}\}$. This defines an action on the module part of $\Delta_W$. 
For $\delta=s_0+s_1$, note that the expression $s_0 = \sum_i (y_i - x_i)\partial_{\theta_i}$ is
$H$-invariant. Taking $H$-averaging of $s_1= \sum_j (\Delta_j W)\theta_j$, we obtain a new $s_1$,
which defines an $H$-equivariant modification $\Delta_W^H$. Observe that if we forget the $H$-action, then it is isomorphic to $\Delta_W$.

Now $\Delta_W^H$ can be used to produce an $H\times H$-equivariant matrix factorization
(analogous to diagonal $\{(x,y) \mid x=y\}$ versus orbifold diagonal $\bigcup_{h \in H} \{ (x,y) \mid hx =  y\}$).

\begin{theorem}\cite{PV}\label{PVtheorem}
Let $\Delta_W$ be a kernel for $\mathrm{Id}_{MF(W)}$ and $\Delta_W^{H}=(R^e[\theta_1,\cdots,\theta_n],d(x,y))$ be its $H$-equivariant modification. Then the following is a kernel for $\mathrm{Id}_{MF_\AI^{H}(W)}$;
\[ \Delta_W^{H\times H}:=\bigoplus_{h \in H}\big(R^e[\theta_1,\cdots,\theta_n],d(h x,y)\big)\]
which is an $H\times H$-equivariant matrix factorization of $W(y) - W(x)$.
 \end{theorem}
 
 \begin{notation}\label{not:moduleaction}
For an $R$-module $M$, Let $rv$ be an element in $M$ with $r\in R$. We distinguish the action on the generator $v$ by $\rho$ as follows:
\[ h\cdot (rv)=(h\cdot r)\rho(h)v.\]
\end{notation}

We explain $H \times H$-equivariant structure of $\Delta_W^{H\times H}$ by describing the action of $h_1\times h_2$. Denote each summand of $\Delta_W^{H\times H}$ by
\[ \Delta_W^h:=\big(R^e[\theta_1,\cdots,\theta_n],d(h x,y)\big).\]
For an element $r(x,y)v\in \Delta_W^h$, 
\begin{equation}\label{eq:hhaction}
 (h_1\times h_2)\cdot (r(x,y)v):= r(h_1 x,h_2 y)\rho(h_2)v \in \Delta_W^{h_1hh_2^{-1}}.
 \end{equation}
Then the $H\times H$-equivariance of $\Delta_W^{H\times H}$ is verified as follows:
\begin{align*}
(h_1\times h_2)\cdot \big(d(h x,y)(r(x,y)v)\big) &= (h_1 h h_2^{-1}\times 1)\cdot (h_2\times h_2) \cdot (h^{-1}\times 1)\cdot \big(d(h x,y)(r(x,y)v)\big) \\
&=(h_1 h h_2^{-1}\times 1)\cdot(h_2\times h_2) \cdot \big(d(x,y)(r(h^{-1} x,y)v)\big) \\
&= (h_1 h h_2^{-1}\times 1)\cdot \Big(d(x,y)\big((h_2\times h_2)\cdot (r(h^{-1}x,y)v)\big)\Big)\\
&=(h_1 h h_2^{-1}\times 1)\cdot \Big(d(x,y)\big((r(h_2 h^{-1} x,h_2 y) \rho(h_2) v)\big)\Big)\\
&= d(h_1 h h_2^{-1}x,y)\big( r(h_1 x, h_2 y)\rho(h_2)v\big)\\
&= d(h_1 h h_2^{-1}x,y)\big( (h_1\times h_2)\cdot (r(x,y)v)\big).
\end{align*}
Note that we essentially used $H$-equivariance of $d(x,y)$ for the third equality.

\begin{remark}\label{rmk:equivkernel}
The description of a kernel for $\mathrm{Id}_{MF_{H}(W)}$ does not depend on the above specific construction of $\Delta_W^H$. The proof of the Theorem \ref{PVtheorem} is based on facts that $\Delta_W^H$ is $H$-equivariant, and that the cokernel of $\Delta_W^H$ is isomorphic to $R^e/(y_1- x_1,\cdots,y_n-x_n)$ as a maximal Cohen-Macaulay module over $R^e/(W(y)-W(x))$. The latter fact implies that the cokernel of $\Delta_W^{H\times H}$ is isomorphic to $\bigoplus_{h \in H}R^e/(y_1-h\cdot x_1,\cdots,y_n-h\cdot x_n)$ as a MCM module. Any matrix factorization $P$ which is also $H$-equivariant and isomorphic to $\Delta_W$ as a nonequivariant matrix factorization can be used to construct a kernel for $\mathrm{Id}_{MF_H(W)}$ in the same way as in Theorem \ref{PVtheorem}. We will appeal to this point later when we encounter a matrix factorization of $W(y)-W(x)$ from Floer theory.
\end{remark} 

\section{Preliminaries on orbifold Jacobian algebras}\label{sec:oja}
Consider a finite abelian group $H$ acting on the commutative algebra $R=k[x_1,\cdots,x_n]$.
 Let $W$ be a $H$-invariant element in $R$.
Then the triple $(R,W,H)$ is called an orbifold Landau-Ginzburg model, and we are interested in closed string theory
of it as a $B$-model.  It is given by an orbifold Jacobian algebra or equivalently defined as Hochschild cohomology of $G$-equivariant matrix factorization category of $W$.

Let us recall the definition of orbifold Jacobian algebra following \cite{Shk}. In fact it has simple description as a $\Z/2$-graded module, but its product structure is quite non-trivial and interesting. Recall that Jacobian algebra of $W$ is $k$-algebra defined by 
$$\Jac(W) = \frac{k[x_1,\cdots,x_n]}{(\frac{\partial W}{\partial x_1}, \cdots, \frac{\partial W}{\partial x_n})}.$$
For mirror symmetry applications, we take $k =\Lambda$. We remark that  $R=\Lambda[ x_1,\cdots,x_n]$ is a Tate algebra and hence any ideal is closed.

We assume $H$ acts on $R$ diagonally. Namely, for $h \in H$, $h \cdot (x_1,\cdots,x_n) = (h_1 x_1,\cdots, h_n x_n)$
for $h_i \in k^*$.
We set $$I^h:= \{i \mid h_i= 1\},\; I_h=\{i \mid h_i \neq 1\},\; d_h = |I_h|.$$
Hence $d_h$ is the codimension of fixed subspace by $h$,  $\mathrm{Fix}(h) \subset k^n$. 

Denote by $W^h$ the restriction of $W$ to $\mathrm{Fix}(h)$.
\begin{definition}\label{def:ojr}\cite{BTW, Shk}
The {\em twisted Jacobian algebra} of $(R,W,H)$ is a $\Z/2$-graded algebra
$$\Jac'(R,W,H) = \bigoplus_{h \in H} \Jac(W^h) \cdot \xi_h $$
where $\xi_h$ is a formal generator of degree $d_h \mod 2$.
Its product structure is defined as follows.
For  $g,h \in H$ and $\phi_{g} \in \Jac(W^g), \phi_{h} \in \Jac(W^h)$, 
we consider their restriction to $\mathrm{Fix}(gh)$ denoted as 
$\phi_{g}^{gh}, \phi_{h}^{gh}$ and define
$$\phi_{g} \xi_{g} \bullet \phi_{h} \xi_{h} := \phi_{g}^{gh}\cdot \phi_{h}^{gh} \cdot \sigma_{g,h} \cdot \xi_{gh}$$
where the definition of $\sigma_{g,h} \in \Jac(W^{gh})$ will be explained in Definition \ref{def:sig}.
We have 
\begin{equation}\label{degcond}
\sigma_{g,h}=0 \;\; \textrm{if} \;\;  d_{g}+d_{h} - d_{gh} \notin 2\Z
\end{equation}
which implies that the twisted Jacobian algebra $\Jac'(R,W,H)$ is $\Z/2$-graded. 
\end{definition}

\begin{definition}\label{def:ojr2}
Define the $H$-action on $\{ \xi_g \mid g\in G\}$ by
\[ h\cdot \xi_g := \prod_{i \in I_h} h_i^{-1}\xi_g \]
where $h=(h_1,\cdots,h_n)$. The {\em orbifold Jacobian algebra $\Jac(R,W,H)$} is the $H$-invariant part of $\Jac'(R,W,H)$.
\end{definition}
\begin{theorem}[\cite{Shk}]
$\Jac(R,W,H)$ is a $\Z/2$-graded commutative algebra.
\end{theorem}

Let us explain the definition of the products between $\xi_h$ following \cite{Shk}.
The construction is based on an explicit relation between bar resolution and Koszul resolution and
we refer readers to the references for the details.  Consider a morphism
\begin{align*}
 \Delta_0: R^{\otimes 2}\otimes k[\theta_1,\cdots,\theta_n] & \to R^{\otimes 3}  \otimes k[\theta_1,\cdots,\theta_n]^{\otimes 2}, \\
 f(x,y)\cdot p(\theta_1,\cdots,\theta_n) & \mapsto f(x,{z})\cdot p(\theta_1\otimes 1+1\otimes \theta_1,\cdots, \theta_n\otimes 1 + 1\otimes \theta_n)
  \end{align*}
and $\Delta:=e^{H_W}\cdot \Delta_0$, where $H_W$ is an element in $R^{\otimes 3}\otimes k[\theta_1,\cdots,\theta_n]^{\otimes 2} $ defined by
\[ H_W(x,y,{z}):=\sum_{1\leq j \leq i \leq n}\nabla_j^{y \to (y,{z})}\nabla_i^{x \to (x,y)}(W) \theta_i \otimes \theta_j.\]
Identifying 
$R^{\otimes 2}[\theta_1,\cdots,\theta_n] \otimes_R R^{\otimes 2}[\theta_1,\cdots,\theta_n] \cong R^{\otimes 3}\otimes k[\theta_1,\cdots,\theta_n]^{\otimes 2} \cong $ via 
\[  f_1(x,y)p_1({\theta}) \otimes f_2(y,{z})p_2({\theta}) \leftrightarrow f_1(x,y)f_2(y,{z})\otimes p_1({\theta})\otimes p_2({\theta}).\]

\begin{definition}\label{def:sig}\cite{Shk}
The structure coefficient $\sigma_{g,h}\in \Jac(W|_{\mathrm{Fix}(gh)})$ of the product between $g$- and $h$-sectors is given by the coefficient of $\partial_{\theta_{I_{gh}}}:=\prod_{i \in I_{gh}} \partial_{\theta_i}$ in the following expression
\begin{equation}\label{eq:twjacprod}
\frac{1}{d_{g,h}!} \Upsilon\big( (\lfloor H_W(x,g\cdot x,x)\rfloor_{gh}+\lfloor H_{W,g}(x)\rfloor_{gh}\otimes 1 + 1\otimes \lfloor H_{W,h}(g\cdot x)\rfloor_{gh})^{d_{g,h}}\otimes \partial_{\theta_{I_g}} \otimes \partial_{\theta_{I_h}}\big). 
\end{equation}
\end{definition}
We explain various notations in \eqref{eq:twjacprod}.
\begin{itemize}
\item $\displaystyle H_{W,g}:= \sum_{i,j\in I_g, j<i}\frac{1}{1-g_j}\nabla_j^{x \to (x,x^g)}\nabla_i^{x \to (x,g\cdot x)}(W)\theta_j \theta_i \in R[\theta_1,\cdots,\theta_n]$ where $x_i^g=x_i$ if $i\in I^g$ and $x_i^g=0$ if $i\in I_g$. The operations $\nabla_i^{x \to (x,g\cdot x)}$ and $\nabla_i^{x \to (x,x^g)}$ are computed by $\nabla_i^{x \to (x,y)}$ followed by substitutions $y=g\cdot x$ and $y=x^g$, respectively.
\item $\lfloor f \rfloor_{g}:= [f|_{\mathrm{Fix}(g)}] \in \Jac(W|_{\mathrm{Fix}(g)}).$
\item $d_{g,h}:=\frac{d_g+d_h-d_{gh}}{2}$. We define $\sigma_{g,h}=0$ if $d_{g,h}$ is not an integer.
\item The map $\Upsilon$ is defined as
\[\Upsilon:R[\theta_1,\cdots,\theta_n]^{\otimes 2}\otimes R[\partial_{\theta_1},\cdots,\partial_{\theta_n}]^{\otimes 2} \to R[\partial_{\theta_1},\cdots,\partial_{\theta_n}],\] 
\[ p_1({\theta})\otimes p_2({\theta})\otimes q_1({\partial_\theta})\otimes q_2({\partial_\theta}) \mapsto (-1)^{|q_1||p_2|}p_1(q_1)\cdot p_2(q_2),\]
where $p_i (q_i)$ means the action of $p_i(\theta)$ on $q_i(\partial_\theta)$ on $\mathrm{Cl}_n / \mathrm{Cl}_n\cdot \langle \theta_1,\cdots,\theta_n \rangle$  via \eqref{eq:cliffrel}.
\end{itemize}
Then the product structure on $\Jac'(R,W,H)$ is defined as follows:
\[ \lfloor f_1\rfloor_g\cdot \xi_g \bullet \lfloor f_2 \rfloor_h\cdot \xi_h:= \lfloor f_1 f_2 \rfloor_{gh}\cdot\sigma_{g,h}\xi_{gh}.\]
We end this section by mentioning the following important theorem.
\begin{theorem}[\cite{Shk}]\label{thm:orbjachh}
Orbifold Jacobian algebra is isomorphic to Hochschild cohomology algebra of $H$-equivariant matrix factorization category.
$$ \Jac(R,W,H)\cong H^*(MF_H(W), MF_H(W)).$$
\end{theorem}

\section{MC formalism of $\AI$-algebra $\mathcal{A}$ and a new $\AI$-algebra $\mathcal{B}$}\label{sec:mc}
Let $\mathcal{A} := (V,\{m_k\})$ is a unital gapped filtered $\AI$-algebra over $\Lambda$ .
In this section, we use  a  Maurer-Cartan formalism of $\mathcal{A}$ \cite{FOOO} to give a definition of  a new $\AI$-algebra $\mathcal{B}$. If $\mathcal{A}$ is weakly unobstructed with a potential function $W$,  $m_1$-cohomology of $\mathcal{B}$ is shown to be isomorphic to Jacobian ring of $W$.

\subsection{Maurer-Cartan setup}
We briefly recall Maurer-Cartan formalism of $\AI$-algebra from \cite{FOOO} to set the notations.
Here $V$ be a $\Z/2$-graded module over $\Lambda_0$ equipped with
$\Lambda_0$-multi-linear $\AI$-operation $m_k:V^{\otimes k} \to V$ of degree $2-k$.
We assume that it is gapped filtered (which means $m_k = \sum_{\beta \in G} m_{k,\beta}$ for a
monoid $G \subset \mathbb{R}_{\geq0}$ such that $\{ x \in G \mid x \leq N\}$ is finite for any $N \in \mathbb{N}$), and $V \otimes_{\Lambda_0} \Lambda$ defines the $A_\infty$-algebra. 
We denote by $\be$ a unit of $\AI$-algebra.
For an element $b \in F_{+}V$ with positive valuation, weak Maurer-Cartan equation is given by
$$ m_0  +m_1(b) + m_2(b,b) + \cdots  = W(b) \be,$$
The sum in the left hand side converges in $T$-adic sense and it is sometimes written as $m(e^b)$.
Given $b_0,\cdots, b_k$, we can deform an $\AI$-operation $m_k$ to
\begin{equation}\label{eq:dd}
m_k^{b_0,\cdots,b_k} (w_1,\cdots,w_k):=\sum_{l_0,\cdots,l_k \geq 0} m_{k+l_0+\cdots+l_k}(b_0^{l_0},w_1,b_1^{l_1},
\cdots,b_{k-1}^{l_{k-1}},w_k,b_k^{l_k}).
\end{equation}
For the case $b_0 = \cdots =b_k=b$, we write $m_k^b := m_k^{b,\cdots,b}$.
Maurer-Cartan equation can be restated as $m_0^b = W(b) \be$.
If $b$ satisfies MC equation, $m_1^b$ defines a chain complex, and this was used
to define Floer cohomology of a Lagrangian $L$ with $m_0 \neq 0$ in \cite{FOOO}.

We will decorate the $\AI$-algebra with formal parameters of deformation as follows.
Take $e_1,\cdots,e_n \in V^{1}.$
Assume that  any $b:= \sum x_i e_i \in \Lambda_+ \langle e_1,\cdots,e_n \rangle$
satisfies the weak MC equation $m(e^b) = W(b) \cdot \be$. Here, $x_i$'s are (formal) dual variables to $e_i$'s
and may be considered as coordinate functions of Maurer-Cartan solution space.

In applications, one may take either $\{x_i\}$ or $\{e^{x_i}\}$ as local mirror coordinates.
In mirror symmetry, the former corresponds to immersed Lagrangians, and the latter for Lagrangian torus.
In this paper, we will develop the theory for the variables $\{x_i\}$. The theory for exponentiated variables
can be constructed analogously following the localized mirror construction of \cite{CHL2} of Lagrangian torus.
We refer readers to Section 3.3 and Lemma 4.2 of \cite{CLS} to see the summary and basic constructions for the case of exponentiated variables.

\subsection{$\mathrm{Jac}(W)$ from MC theory}
Now, we set $k = \Lambda$ and still denote by $R= \Lambda[x_1,\cdots,x_n]$   the completion of the polynomial ring with respect to the valuation $\nu$ of $\Lambda$.

For a gapped filtered $\AI$-algebra $\mathcal{A} = (V, \{m_k\})$, we assume that at least its cohomology is finite dimensional:
If $\Lambda_0$-module $V$ itself is not finitely generated, we take its canonical model following \cite{M} and \cite{FOOO_can}.
We choose a finite dimensional subspace $H \subset V$, $\pi:V \to H, i:H \to V$ as well as the contraction homotopy $Q$
such that $\mathrm{id} - \pi = -(m_1 Q +Q m_1)$. Then, it is well known that we can transfer the $\AI$-structure on $H$ and
find $\AI$-quasi-isomorphisms between them in a combinatorial way. 

\begin{definition}\label{def:B}
We take a  tensor product $V  \otimes_{\Lambda_0} R$.
By linearly extending $\AI$-operation over formal variables $x_1,\cdots,x_n$ (with $b = \sum_{i=1}^n x_i X_i$), 
 $\{m_k^b\}$ defines an $\AI$-algebra on $V \otimes_{\Lambda_0} R  $. We denote $$ \mathcal{B} := (V \otimes_{\Lambda_0} R, \{m_k^b\}).$$
 If $b$ satisfies the Maurer-Cartan equation $m_0^b= W(b) \cdot \be$, and  $(\mathcal{B},m_1^b)$ defines a chain complex, and
 $m_2^b$ defines a product on its cohomology
\end{definition}
\begin{remark}
Cohomology of $\AI$-algebras $\CA$ and $\CB$ are quite different.
In mirror symmetry applications, $\CA$ is $\AI$-algebra for a Lagrangian $L$ (open string theory), but
associated $m_1^b$ cohomology $\CB$ will be Jacobian ring of the potential $W_L$ (closed string theory).
\end{remark}
\begin{remark}
The process of taking canonical model, and tensoring $R$ can be taken at the same time (which was used in Section 4.2 \cite{ACHL}).
Namely, we may carry out the same construction of
canonical model for $H \otimes \Lambda_0[x_1,\cdots,x_n] \subset V \otimes \Lambda_0[x_1,\cdots,x_n]$ and
transfer $\AI$-structure $\{m_k^b\}$ on the latter to the former. We can construct $\AI$-quasi-isomorphisms between them as before.
\end{remark}

We find a sufficient condition that  cohomology of $\CB$ becomes Jacobian ring.
We first define $e_I, e_I^b$ for a subset $I \subset \{1,\cdots,n\}$, which are successive $m_2$ or $m_2^b$ products.
For any  $I=\{i_1,\cdots,i_k\}$ with $i_1 < i_2 < \cdots <i_k$,  we set $e_\emptyset = e_\emptyset^b = \be$ and 
denote 
\begin{eqnarray}\label{eq:11}
e_I &=&m_2(e_{i_1},m_2( e_{i_2}, m_2( \cdots, e_{i_k}) \cdots )) \\
e_I^b &=& m_2^b(e_{i_1},m_2^b( e_{i_2}, m_2^b( \cdots, e_{i_k}) \cdots )).
\end{eqnarray}

\begin{assumption}\label{as1}
$R$-module $V  \otimes_{\Lambda_0} R$ is generated by $\{e_I^b\}_{I \subset \{1,\cdots n\}}$.
\end{assumption}

We compare the generation of $e_I$ and $e_I^b$.
From the gapped filtered condition (with the valuation $\nu$ on $\Lambda$-modules) we obtain the following lemma.
\begin{lemma}\label{lem:ge}
Suppose $\Lambda$-module $V \otimes_{\Lambda_0} \Lambda$ is generated by $\{e_I\}_{I \subset \{1,\cdots n\}}$
and that  
$$\nu( m_2^b(v,w) - m_2(v,w) ) > \epsilon$$ for some positive $\epsilon>0$ for any $v,w$.
Then $V  \otimes_{\Lambda_0} R$ satisfies the Assumption \ref{as1}.
\end{lemma}

\subsection{Examples satisfying Assumption \ref{as1}}
\begin{lemma}
Any  Lagrangian torus $L$ (at critical points of the potential $W_L$) satisfies the Assumption \ref{as1}.
\end{lemma}
\begin{proof}
We will show that $e_I$ generate $V=H^*(L,\Lambda)$ and use it to show that $e_I^b$ generate $V  \otimes_{\Lambda_0} R$.
Let $L$ be a Lagrangian torus, with exponential coordinates $z_i = e^{x_i}$ on MC space.
Assume that the potential function $W_L(z)$ has a critical point at $(1,\cdots,1)$ (which can be achieved for
any critical point by change of coordinates).
It is well-known that in this case $L$ with $b=0$ has non-trivial Floer cohomology, isomorphic to singular cohomology of $L$  
(see \cite{CHL2} for example).
For degree one generators $e_1,\cdots, e_n$ of singular cohomology, $m_{2,0}$ products generate the whole cohomology
classes, and therefore so do their $m_2$ products (because of the filtration).
We argue that their $m_2^b$ product generate $H^*(L;\Lambda) \otimes \Lambda[x_1,\cdots,x_n]$.
To see this, first we may take a canonical model $m_k^{can}$ of $\AI$-algebra. We may further assume that classical part $m_{k,0}^{can}=0$ for $k \geq 3$ since $L$ is a torus. Here we write $m_k = m_{k,0} +m_{k,+}$ where $m_{k,0}$ is the classical $\AI$-structure on $L$,
and $m_{k,+}$ are defined using non-constant holomorphic discs.
Then, $m_{2,0}^b(v,w) = \sum m_{k,0}(e^b,v,e^b,w,e^b) = m_{2,0}(v,w)$. Therefore, $e_I^b-e_I$ has strictly positive energy. 
We apply the previous lemma to obtain the claim.
\end{proof}

\begin{lemma}\label{lem:sa1}
The $\AI$-algebra of Seidel Lagrangian \cite{S} in $\mathbb{P}^1_{a,b,c}$ satisfies the above assumption.
\end{lemma}
\begin{proof}
We check the assumptions of Lemma \ref{lem:ge} in this case.
Recall that we consider an immersed Lagrangian $\bL$ (called Seidel Lagrangian) in the orbi-sphere $\PP^1_{a,b,c}$.
First, $m_2$ products for Seidel Lagrangian are computed in \cite{S}:
Recall that $CF(\bL,\bL)$ is generated by six immersed sectors $X_1, X_2, X_3$,  $\bar{X}_1, \bar{X}_2, \bar{X}_3$ as well as two Morse generators $e,p$. We have  $m_2(X_i,X_{i+1}) = T^c \bar{X}_{i+2}=-m_2(X_{i+1},X_{i})$
for the area $c$ of minimal triangle. Also, $m_2(X_i,\bar{X}_i) = p = -m_2(\bar{X}_i, X_i)$, and $e$ is the unit for
$m_2$ multiplication. Other $m_2$ products are shown to vanish by grading considerations (it is shown for $(5,5,5)$ but it is
straightforward to generalize it for general $(a,b,c)$).   Therefore $\{e_I\}$ generate $V$.
Note that only non-trivial holomorphic curve contribution comes from minimal triangle enclosed by Seidel Lagrangian.
It is immediate that any other holomorphic polygons has strictly larger area. Therefore,
we have $\nu( m_2^b(v,w) - m_2(v,w)) > \epsilon$ for any $v,w$ for some $\epsilon<c$. This proves the lemma. 
\end{proof}

Let us further assume that $W$ has isolated singularity at $0$. i.e. $(\partial_{x_{1}}W, \cdots, \partial_{x_n}W)$ defines a regular sequence and an associated Koszul complex
$(K^\bullet, s_0)$ from \eqref{eqK}.
\begin{prop}\label{prop:bjac}
Suppose an $\AI$-algebra $\CA$ satisfies the Assumption \ref{as1}.
The chain complex $(V\otimes R, m_1^b)$ is isomorphic to the Koszul complex  $(K^\bullet, s_0)$ for the regular sequence
$(\partial_{x_{1}}W, \cdots, \partial_{x_n}W)$. More precisely, 
a  chain map 
$$\Psi : (V\otimes_{\Lambda_0} R, m_1^b) \to (K^\bullet, s_0)$$
defined by 
$$ \Psi(e_I^b)= \theta_{i_1} \cdots \theta_{i_k}$$
gives an isomorphism.
Moreover $m_1^b$-cohomology of $\mathcal{B}$ with the product given by $m_2^b$, is isomorphic to $\mathrm{Jac}(W)$
as an algebra.

\end{prop}
\begin{remark}
If $W$ does not have isolated singularity, $\Psi$ defines an isomorphism to Koszul cohomology.
But this may not be an algebra isomorphism.
\end{remark}
\begin{proof}
By taking $\frac{\partial}{\partial x_i}$ of the equation $m(e^b) = W(b) \cdot \be$, we obtain
$$m_1^{b}(e_i) = \partial_{x_i} W \cdot \be$$
Recall that $m_2( \be, \star) = (-1)^{|\star|} m_2(\star,\be) = \star$. 
Then the claim follows from Leibniz rule for $m_1^b,m_2^b$.
For example, consider the identity
\begin{align*}
&m_1^b (e_{i_1 i_2 \cdots i_k}) \\
=& - m_2^b\big(m_1^b(e_{i_1}),m_2^b( e_{i_2}, m_2^b( \cdots, e_{i_k}) \cdots )\big)
-  m_2^b\big(e_{i_1}, m_1^b(m_2^b( e_{i_2}, m_2^b( \cdots, e_{i_k}) \cdots ))\big)
\end{align*}
where the first term equals $ (- \partial_{x_{i_1}} W ) m_2^b( e_{i_2}, m_2^b( \cdots, e_{i_k}) \cdots )$.
In this way, we can prove the claim inductively. 

For the products, we show that the induced product structure on cohomology are the same. 
Recall that the homology of the Koszul complex, $\mathrm{Jac}(W)$
is concentrated at the $0$-th wedge product, and the product structure of Jacobian ring is induced from that of $R$.
Correspondingly, $m_1^b$-cohomology comes from $R \cdot e_{\emptyset} = R \cdot \be$. 
Then $m_2^b$ is just given by the products of the coefficients in $R$. This proves the proposition.
\end{proof}
Thus, $\mathcal{B}$ provides a chain complex model for Jacobian ring of $W$. 
In Section  \ref{sec:KS}, we will define a Kodaira-Spencer  ring homomorphism from
quantum cohomology to $m_1^b$ cohomology of $\CB$.

\section{Kernel from MC formalism of Floer theory}\label{sec:le}
In this section, we will propose a new assumption, the Assumption \ref{as2}.  This identifies Floer theoretic construction of
matrix factorization and the kernel for matrix factorization category. Therefore we call the former Floer theoretic Kernel.
 We show that monotone Lagrangian torus as well as  Seidel Lagrangian in the orbisphere $\mathbb{P}^1_{a,b,c}$
satisfy the Assumption \ref{as2}.

Let us first construct Floer  Kernel. For an $\AI$-algebra $\mathcal{A} = (V,\{m_k\})$
we consider two set of Maurer-Cartan elements  which are the same except the naming of variables.  
$$b(x) = \sum_{i=1}^n x_i e_i, \;b(y) = \sum_{i=1}^n y_i e_i.$$
We denote by $R = \Lambda[x_1,\cdots,x_n], R'=\Lambda[y_1,\cdots,y_n]$ and 
$$R^e=R \hat{\otimes}_\Lambda R' = \Lambda[x_1,\cdots,x_n,y_1,\cdots,y_n]$$
 Let us also write $W(y) = W(b(y)), W(x) =W(b(x))$.
Then, we have degree one operation $m_1^{b(x),b(y)}$  (defined in \eqref{eq:dd}) on $\Z/2$-graded finitely generated $R^e$-module $V \otimes_{\Lambda_0}R^e$ satisfying $$\big(m_1^{b(x),b(y)}\big)^2 = (W(y) - W(x))\cdot \id.$$
Let us call this matrix factorization a Floer kernel  $\Delta_{FL}$.
\begin{assumption}\label{as2}
We assume that
$$\big(V \otimes_{\Lambda_0}R^e[n], m_1^{b(x),b(y)} \big) \;\; \cong \;\;
\big( R^e[\theta_1,\cdots,\theta_n], \sum_i (y_i - x_i) \partial_{\theta_i} +  \sum_i \nabla_i^{x \to (x,y)}W \cdot \theta_i\big).$$
Namely, Floer kernel $\Delta_{FL}$  is quasi-isomorphic to a Koszul matrix factorization $\Delta_W$ of $W(y) - W(x)$ (see \eqref{eq:ksmf}) up to shift of $\Z/2$ grading by $[n]$. 
\end{assumption}
\begin{remark}\label{rem:as2}
Here is an equivalent version without the shift $[n]$.
\begin{equation}\label{eq:mfkz}
\big(V \otimes_{\Lambda_0}R^e, m_1^{b(x),b(y)} \big) \;\; \cong \;\; \big( R^e[ \partial_{\theta_1},\cdots, \partial_{\theta_n}], \sum_i (y_i - x_i) \partial_{\theta_i} +  \sum_i \nabla_i^{x \to (x,y)}W \cdot \theta_i\big).
\end{equation}
We use this in Section \ref{sec:ag} for comparison of orbifold Jacobian algebra and Floer theory
\end{remark}

Let us give another conjectural explanation of the above assumption.
If a symplectic manifold $M$ is mirror to $W$,  one may expect that  the product symplectic manifold $M^- \times M$ is mirror to $-W(x)+W(y)$.
The diagonal Lagrangian $\Delta_M \subset M^- \times M$ gives identify functor on Fukaya category in the language of Lagrangian correspondence. Hence it is natural to expect that $\Delta_M$ and $\Delta_W$ should be mirror to each other
under homological mirror symmetry.

Given a localized mirror functor $\mathcal{F}^\bL:\mathrm{Fukaya}(M) \to MF(W)$, there should be a localized mirror functor for
the product
$\mathcal{F}^{\bL \times \bL}$  that send the diagonal $\Delta_M$ in Fukaya category to the diagonal $\Delta_W$ in matrix factorization
category. Now, let us explain how it is related to the Assumption \ref{as2}.
The functor $\mathcal{F}^{\bL \times \bL}$ sends $\Delta_M$ to
the following decorated Floer complex (which becomes a matrix factorization)
$$\mathcal{F}^{\bL \times \bL}(\Delta_M) = \big(CF(\Delta, \bL \times \bL)\otimes R^e, m_1^{0, b(x)\otimes 1 + 1\otimes b(y)}\big).$$
We conjecture that this is isomorphic to matrix factorization $\big(CF(\bL,\bL)\otimes R^e, -m_1^{b(x),b(y)}\big)$. We hope to
explore this in more detail elsewhere.

Therefore, we expect  this assumption \ref{as2} to hold in general but we are only able to check this for two set of examples in the rest of this section.

\subsection{Monotone Lagrangian torus satisfy Assumption \ref{as2}}
\begin{prop}
Fukaya $\AI$-algebra for a monotone Lagrangian torus satisfies Assumption \ref{as2}.
\end{prop}
\begin{proof}
We prove this lemma following Theorem 9.1 \cite{CHL}.
Let us recall that $\AI$-algebra for monotone Lagrangian torus $L$ has the following structure.
First, without the holomorphic disc contribution, the classical cochain algebra is quasi-isomorphic to the exterior algebra
$\wedge^\bullet W$ for $$W = H^1(L,\Lambda_0) = \Lambda_0 \langle e_1,\cdots,e_n \rangle.$$  Therefore we can transfer its $\AI$-structure to $( \bigwedge \nolimits^\bullet W, \{m_k\})$  such that
 $$m_k =\sum_{\mu \in 2 \Z, \mu\geq 0} T^\mu m_{k,\mu}. $$
 $m_{1,0} = m_{k,0} =0$ for $k \neq 2$, and $(-1)^{|u_1|}m_{2,0}(u_1,u_2)$ for $u_1,u_2 \in \bigwedge\nolimits^\bullet W$ corresponds to
 the exterior algebra structure.
Here $m_{k,\mu}$ records the quantum contribution from Maslov index $\mu$ holomorphic discs, and monotone condition guarantees that $\mu >0$ for non-trivial holomorphic disc contributions.

This $\AI$-algebra is only $\Z/2$-graded, but in terms of the degree of exterior algebra, $m_{k,\mu}$ has degree $2-k - \mu$.
For example, we have $m_1 = m_{1,0} + T m_{1,2}  + T^2 m_{1,4} + \cdots$
where $$m_{1,\mu}: \bigwedge \nolimits^\bullet W \to \bigwedge \nolimits^{\bullet +1 - \mu} W.$$
Hence, wedge grading is an enhancement of $\Z/2$-grading.
Also, from weak Maurer-Cartan equation, the potential $W(b)$ comes from Maslov index two disc contribution (namely, $\{m_{k,2}\}$)

Let us assume $n$ is even so that the shift $[n]$ is trivial. The case of odd $n$ is similar and omitted.
We set
$$m_{1,\mu}^{b(x),b(y)}(w):= \sum m_{k+1+l, \mu} \big( \underbrace{ b(x), \cdots, b(x)}_k, w, \underbrace{b(y), \cdots, b(y)}_l \big).$$
Then, degree of $m_{1,\mu}^{b(x),b(y)}(w)$ is the same as $\deg(w) +1-\mu$, and we have
$$m_{1,0}^{b(x),b(y)}(e_I)= m_{2,0}(b(x),e_I) + m_{2,0}(e_I,b(y)) =  \sum_i (y_i - x_i) e_i \wedge e_I.$$
Therefore, $m_{1,0}^{b(x),b(y)}$ is given by wedge operation $\sum_i (y_i - x_i) e_i \wedge \cdot$.

Now, consider the shift $[n]$ in $\Z/2$-grading of  $\big(\wedge^\bullet W \otimes_{\Lambda_0}R^e[n], m_{1,0}^{b(x),b(y)} \big)$,
which can be realized by the suitable dual complex.
Namely, we consider the same complex
\begin{equation}\label{kn}
 \big(\wedge^\bullet W \otimes_{\Lambda_0}R^e, d_{1,0}^{b(x),b(y)} \big)
\end{equation}
but we replace wedge operation of $m_{1,0}^{b(x),b(y)}$ by contraction operation and denote it by
$$d_{1,0}^{b(x),b(y)} = \sum_i (y_i - x_i) \iota_{e_i}.$$
We denote by $d_{k,\mu}$ the operations corresponding to $m_{k,\mu}$.
In this way, \eqref{kn} defines a Koszul complex for the
regular sequence $(y_1-x_1,\cdots, y_n-x_n)$. Let $I$ be the ideal of $R^e$ generated by this regular sequence.

Recall that for a matrix factorization $(P^\bullet, \delta)$ of $W$, the corresponding sheaf (via Orlov equivalence) in the singularity category is given by the cokernel of $\delta:P^{odd} \to P^{even}$. Polishchuk-Vaintrob computes this cokernel of Koszul matrix factorization using a spectral sequence in Proposition 2.3.1 \cite{PV}. We will use this spectral sequence following \cite{CHL2} to prove our claim.
The bi-complex $L^{\bullet,\bullet}$ is concentrated on two diagonals $i+j=0$ and $i+j=-1$ given by 
$$L^{-i,i} = \bigwedge \nolimits^{2i} W, L^{-i,i-1} = \bigwedge \nolimits^{2i-1} W$$
with the differentials $$d_{1,\mu}^{b(x),b(y)}:  L^{-i,i-1} \to L^{-(i-1+\frac{\mu}{2}),i-1+\frac{\mu}{2}}.$$
 Consider the spectral sequence of this bi-complex coming from horizontal filtration. It is shown in \cite{PV} that this spectral sequence
 has an $E_2$-page (after taking homology with  $d_{1,0}^{b(x),b(y)}$) with $E_2^{0,0} = R^e/I$,  $E_2^{i,j} =0$ for $i\neq -j$, and $E_2^{-i,i}$ becomes zero in the singularity category for $i \neq 0$.
Hence, the cokernel is isomorphic to $R^e/I$.
Note that higher Maslov index contribution  $d_{1,\mu}^{b(x),b(y)}$ with $\mu \geq 4$ becomes trivial
since $E_2$-page is only non-trivial on the diagonal $ i+j=0$. Therefore our spectral sequence degenerates at $E_2$-page also and  the cokernel for the matrix factorization
$ \big(\wedge^\bullet W \otimes_{\Lambda_0}R^e[n], m_1^{b(x),b(y)} \big)$ is isomorphic to $R/I$ in the singularity category.
Hence it is isomorphic to the desired Koszul matrix factorization.
\end{proof}

In fact, we can strengthen Theorem 9.1 \cite{CHL} if we apply the similar argument as in  the above proof 
and we write the result for reader's convenience.

\begin{prop}
For a monotone Lagrangian torus $L$ with a potential function $W$, the localized mirror functor of $L$
sends $L$ to the Koszul matrix factorization of $W$ obtained by classical and Maslov index two contributions
\end{prop}

\subsection{Orbi-spheres $\mathbb{P}^1_{a,b,c}$ satisfy Assumption \ref{as2}}
The Seidel Lagrangian $\bL$ in the orbi-sphere $\PP^1_{a,b,c}$ 
with bounding cochain $b=x_1X_1+x_2 X_2+x_3 X_3$  for three  degree one immersed  sectors $X_1, X_2,X_3$
are shown to be  weakly unobstructed and localized mirror functor with reference $\bL$ provides homological
mirror symmetry in elliptic and hyperbolic cases \cite{CHL}. Recall that we equip $\bL$ a complex line bundle with holonomy $(-1)$
which is uniformly distributed (as in \cite{CHLnc}). We prove
\begin{prop}\label{prop:abckz}
Fukaya $\AI$-algebra for the Seidel Lagrangian $\bL$ in  $\PP^1_{a,b,c}$, satisfies Assumption \ref{as2}.
\end{prop}
\begin{proof}
We will show that  the matrix factorization $CF\big((\bL,b(x)),(\bL,b(y))\big)$ is a Koszul matrix factorization
for $(y - x,\vec{w})$ for some $\vec{w}= (w_1,\cdots,w_n)$ (see \eqref{eqK}).
As mentioned in the proof of Lemma \ref{lem:sa1},
$CF(\bL,\bL)$ has 8 generators and its $m_2$ product structure may be identified with an exterior algebra $\bigwedge^\bullet \langle X_1,X_2,X_3 \rangle$ by setting $T^c \bar{X}_{i+2} = X_{i} \wedge X_{i+1}$ for $i = 0,1,2 \mod 3$ and $T^c p = X_1 \wedge X_2 \wedge X_3$.
It turns out we need a  quantum correction to this identification to match with a Koszul matrix factorization.
\begin{prop}\label{prop:ge}
There exist $\gamma^e \in \Lambda[x_1,x_2,x_3,y_1,y_2,y_3]$ (to be defined in \eqref{def:ge}) such that
in terms of a basis 
\begin{equation}\label{eq:basis}
\{ p^{new},X_1,X_2,X_3,e,\bar{X_1}^{new},\bar{X_2}^{new},\bar{X_3}^{new}\},
\end{equation}
with $\bar{X_i}^{new}:=\gamma^e \bar{X_i}$ and $p^{new}:=\gamma^e p$,
the map $m_1^{b(x),b(y)}$ on $CF\big((\bL,b(x)),(\bL,b(y))\big)$
can be written as the following linear transformation
\begin{equation}\label{mfabc}
\left(\begin{array}{cccccccc}  
0& 0  &0   &0   & 0  & y_1-x_1  & y_2-x_2   & y_3-x_3  \\  
0&  0 &0   & 0  &  y_1-x_1 & 0  &  c_{12} & c_{13}  \\  
0&  0 & 0  &  0 &  y_2-x_2 & -c_{12}   & 0 & c_{23}  \\  
0&  0 &  0 &  0 &  y_3-x_3 & -c_{13}  & - c_{23}  & 0 \\  
0 &  f_1 & f_2  & f_3  & 0  & 0  & 0  & 0  \\  
f_1&  0 & y_3-x_3  & -(y_2-x_2)  &  0 &  0 &  0 & 0  \\  
f_2&  -(y_3-x_3) &  0 & y_1-x_1  &  0 &  0 & 0  &  0 \\  
f_3&   y_2-x_2& -(y_1-x_1)  &0   &  0 & 0  &  0 & 0  \end{array}\right).
\end{equation}
\end{prop}
\begin{remark}\label{rem:ga}
Recall that in \cite{CHL}, the matrix factorization of $W(y)$ that is mirror to $\bL$ was
shown to be Koszul. We also needed quantum change of variables for $CF\big(\bL ,(\bL,b(y))\big)$ with $\gamma$ 
playing the role of $\gamma^e$ here. In fact, we have $\gamma^e \mid_{x_1=x_2=x_3=0} = \gamma$.
\end{remark}
\begin{corollary}
We have $$c_{12} = f_3, c_{23}=f_1, c_{31}= f_2.$$
Therefore,  $CF\big((\bL,b(x)),(\bL,b(y))\big)$ defines a Koszul matrix factorization
for  $(y - x,\vec{w})$ with $w_i= f_i = c_{i+1,i+2}$ for $i=1,2,3 \mod 3$.
\end{corollary}
\begin{proof}
Let us first prove the corollary. Using the fact that \eqref{mfabc} defines a matrix factoriztion for $W(y) - W(x)$, we
obtain that
$$W(y) - W(x) = \sum_{i=1}^3 (y_i - x_i) f_i = \sum_{i=1}^3 (y_i - x_i) c_{i+1,i+2}$$
and 
$$- f_2 c_{12} + f_3 c_{31} =0, f_1 c_{12} - f_3 c_{23}=0, -f_1 c_{31} + f_2 c_{23}=0$$
\end{proof}
\begin{proof}
Let us prove Proposition \ref{prop:ge}.
 By unital property of $\AI$-algebra
$$ m_1^{b(x),b(y)}(e) = m_2( e, b(y)) +m_2 (b(x),e) =  b(y) - b(x)$$ and this explains the case of input $e$.
The map from $\bar{X_i}$ to $p$ has only constant disc (with Morse trajectory) contribution, and we omit the details.

For the rest of the proof, reflection symmetry along the equator of $\mathbb{P}^1_{a,b,c}$  plays the main role. Recall that $\bL$ is preserved by reflection and $e,p$ are
switched to each other.
The coefficients from $X_i$ to $e$ is the same as the coefficient from $p$ to $\bar{X_i}$ using reflection argument in (iv) of Theorem \cite{CHL}.
Furthermore, we have the following skew-symmetry from reflection.
\begin{lemma}
Let $c_{ij}$ be the coeffcient of $X_j$ in  $m_1^{b(x),b(y)}(\gamma^e \bar{X}_i)$.
We have $c_{ij} = - c_{ji}$.
\end{lemma}
\begin{proof}
We check the sign following Lemma 7.4 \cite{CHL}.
Let $P$ be a polygon contributing to the coeffcient $c_{ij}$, and then its reflection image $P^{op}$ contribute to
the coefficient $c_{ji}$ (see Figure \ref{fig:refsymmetry}).
\begin{figure}
\includegraphics[height=2in]{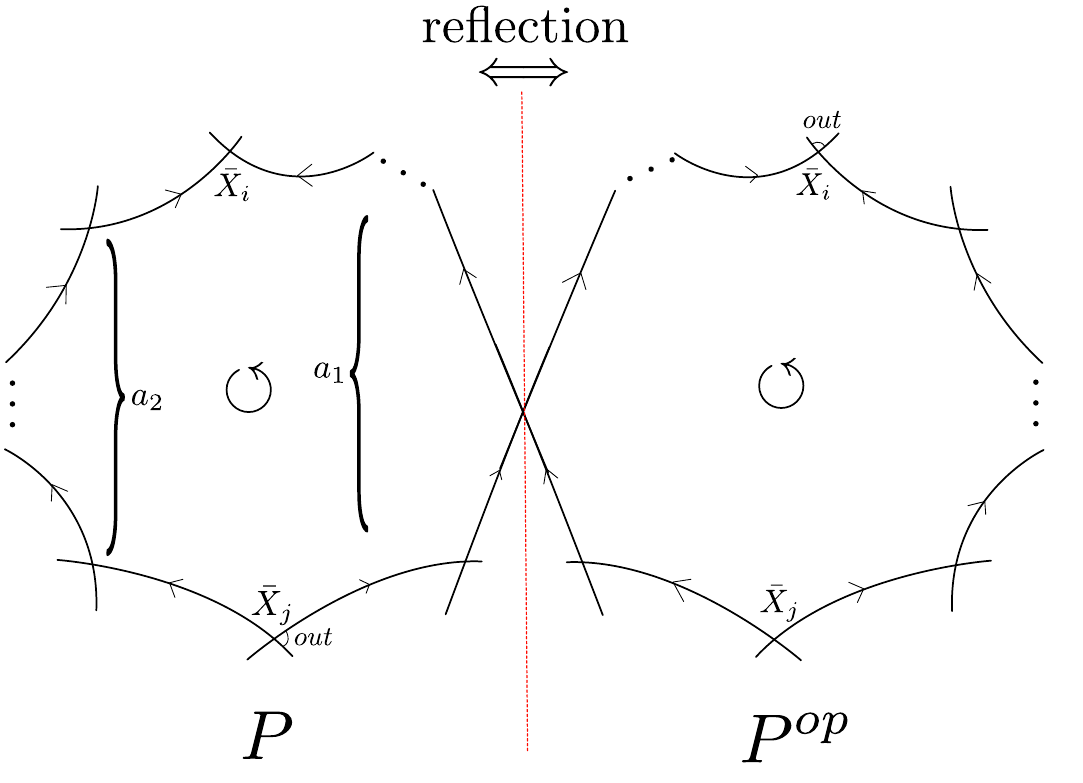}
\caption{Reflection symmetry of holomorphic polygons.}
\label{fig:refsymmetry}
\end{figure}
As observed in Lemma 7.4 \cite{CHL}, $\partial P$ and $\partial P^{op}$ evenly covers $\bL$, say $l$ times.
Since $\bL$ consists of 6 minimal edges, we may assume that $\partial P$ and $\partial P^{op}$ each covers $3l$ minimal edges.

For polygon $P$, denote by $a_1$ (resp. $a_2$) the number of corners that lies between the output corner to the input corner
when we walk along $\partial P$ counter-clockwise (resp. clockwise) way. 
From the combinatorial sign convention, it is easy to see that the sign difference of $\AI$-operation for $P$ and $P^{op}$ differ
by $(-1)^{a_1+a_2+1}$ :  If $\bL$ is oriented as in the Figure, then $P$ carries signs $(-1)^{a_2+1}$ and $P^{op}$ carries signs $(-1)^{a_1}$.
Extra $(-1)$ factor for $P$ comes from the output degree $|X_j|$.

One can also observe that for the holomorphic polygon $P$, the length of minimal edge between corners of same parity is odd, and the corners of different parity is even.
Therefore, we can see that the parity of the number of edges of $P$ is given by $(a_1-1)+(a_2-1) = a_1+a_2$.
Hence, $a_1+a_2 \equiv l$ modulo 2. 

The additional sign difference from the non-trivial spin structure of $\bL$ for $P$ and $P^{op}$ is given by
$(-1)^l = (-1)^{a_1+a_2}$. Combining these two contributions  $(-1)^{a_1+a_2+1} \times (-1)^{a_1+a_2}$, we obtain the result.
\end{proof}
The following lemma can be proved in a similar way and we omit the proof.
\begin{lemma}\label{lem:sqs}
Denote by $h_{ij} \in R^e$ the coefficient of $\bar{X}_j$ in  $m_1^{b(x),b(y)}(X_i)$ as $h_{ij}$.
Then, $h_{ij} = - h_{ji}$.
\end{lemma}

Now, let us explain the definition of $\gamma^e$. 
The coefficient of $p$ in $\big(m_1^{b(x),b(y)}\big)^2 (X_i) $ should vanish (from $\AI$-identity) for $i=1,2,3$ and this gives
\begin{eqnarray*}
(y_2 - x_2) h_{12} + (y_3 - x_3) h_{13} &=& 0, \\
(y_1 - x_1) h_{21} + (y_3 - x_3) h_{23} &=& 0, \\
(y_1 - x_1) h_{31} + (y_2 - x_2) h_{32} &=& 0. \\
\end{eqnarray*}

Using Lemma \ref{lem:sqs}, we have the following equality and we denote it by $\gamma^e \in R^e$.
\begin{equation}\label{def:ge}
 \frac{h_{12}}{y_3-x_3} = \frac{h_{23}}{y_1-x_1}  =\frac{h_{31}}{y_2-x_2} =: \gamma^e.
\end{equation}
Here $h_{12}$ is divisible by $y_3-x_3$ by the above $\AI$ equation.
Therefore, the coefficient of $\bar{X_2}^{new} (= \gamma^e\bar{X_2})$ in  $m_1^{b(x),b(y)}(X_1)$ is just $(y_3-x_3)$
and so on. This proves the lemma.

\end{proof}
Thus, identifying the basis \eqref{eq:basis} with the basis of exterior algebra, we
have shown that the matrix factorization $m_1^{b(x),b(y)}$ is 
a Koszul MF for $(y - x, \vec{f})$.
After shifting $[3]$ in $\Z/2$-grading, we obtain a Koszul MF for $( \vec{f}, y - x)$,
hence obtaining a Koszul resolution of $W(y) -W(x)$ for the diagonal $\Delta$.
\end{proof}

A priori, $\gamma^e \in R^e$ and $\gamma \in R$ are different. But for $\mathbb{P}^1_{3,3,3}$ they are the same.
\begin{lemma}
 For $\PP^1_{3,3,3}$,   $\gamma^e = \gamma$ is a scalar in $\Lambda$ given by a modular form given in \eqref{eq:g}.
\end{lemma}
\begin{proof}
One can check that $h_{ij}$ is linear, hence $\gamma^e$ is scalar. Therefore it equals $\gamma$ by the remark \ref{rem:ga}.
\end{proof}

\section{Equivariant construction and $\AI$-algebra $(\mathcal{B} \rtimes \WH{G})^{\WH{G}}$}\label{sec:eqb}
In this section, we give an equivariant construction of Section \ref{sec:mc}.
When a finite abelian group $G$ acts on an $\AI$-category, we  look at the quotient $\AI$-category and for a reference object $\OL{O}$,
we construct a new $\AI$-algebra $\mathcal{B}$ by studying Maurer-Cartan equation in the quotient.
Then we construct the theory for the original category by developing an equivariant construction for the dual group $\WH{G}$-action

Let $\mathcal{C}$ be a filtered unital  $\Z/2$-graded $\AI$-category over $\Lambda$ with a strict $G$-action. This means that $G$ acts on the set of objects $Ob(\mathcal{C})$
and morphisms such that 
\[g:hom_{\mathcal{C}}(O_1,O_2) \to hom_{\mathcal{C}}(gO_1,gO_2),\] and for composable morphisms $w_1,\cdots,w_k$,
we have
$$m_k(g w_1,\cdots,g w_k) = g m_k(w_1,\cdots,w_k).$$

We introduce the following notation.
\begin{definition}
Define $\WT{O} := \bigoplus_{g \in G} gO$ for any object $O$ of $\mathcal{C}$, and
$$hom_{\CC} (\WT{O}, \WT{O}) :=  \bigoplus_{g_1,g_2 \in G} hom_{\mathcal{C}}(g_1O, g_2O).$$
This has an induced $\AI$-structure $\{m_k\}$ where $m_k$ is defined to be $0$ if not composable.
It has the strict diagonal $G$-action.
\end{definition}
One can define a (quotient) $\AI$-cateory $\mathcal{C}^G$ as follows.
Let us discuss this for the case of single object for simplicity and 
let $O \in Ob(\mathcal{C})$ be an object with $g_1 O \neq g_2O$ for $g_1 \neq g_2$.
\begin{definition}
We define an $\AI$-algebra $\mathcal{A}$ (for the quotient object $\OL{O}$)  on
$$hom_{\mathcal{C}^G}(\OL{O}, \OL{O}) := \bigoplus_{g \in G} hom_{\mathcal{C}}(O,gO).$$
Let us denote an element $v \in hom_{\mathcal{C}}(O,gO)$ as $v_g$ to keep track of the indices. An $\AI$-structure on $hom_{\mathcal{C}^G}(\OL{O}, \OL{O})$ is defined as
\begin{equation}
m_k( (w_1)_{g_1}, \cdots, (w_k)_{g_k})
= m_k(w_1,g_1\cdot(w_2), (g_1g_2) \cdot (w_3), \cdots, (g_1,\cdots,g_{k-1}) \cdot w_k)
\end{equation}
\end{definition}
One may check that the above $m_k$ operation is composable, and satisfies $\AI$-equations.

Let $\WH{G}=Hom(G,U(1))$ be the character group of the finite abelian group $G$.
We define $\WH{G}$-action on quotient morphism spaces. 
\begin{definition}\label{def:hatac0}
We  define $\WH{G}$-action on $hom_{\mathcal{C}^G}(\OL{O}, \OL{O})$ by
\begin{equation}\label{def:cocycleaction}\chi (v_g) = \chi(g^{-1}) v_g, \;\; v_g \in hom_{\mathcal{C}}(O,gO).\end{equation}
\end{definition}
\begin{remark}
In \cite{S1}, action was defined to be $\chi(g)v_g$. Our convention has the advantage that 
the $\AI$-isomorphism in  Lemma \ref{lem:qq} preserves the eigen-spaces of $G$-action.
\end{remark}

Let us denote the action by $\rho(\chi)$. We may call this the first $\WH{G}$-action, which  we use to define the following semi-direct product. Later, the second $\WH{G}$-action will be defined in Definition \ref{def:hatac}.

\begin{definition}\cite{S1}
A semi-direct product $\AI$-algebra $$hom_{\mathcal{C}^G}(\OL{O}, \OL{O}) \rtimes \WH{G}$$
is an $\AI$-structure defined on 
$hom_{\mathcal{C}^G}(\OL{O}, \OL{O}) \otimes \Lambda[\WH{G}]$ with its $\AI$-operation is defined as 
\begin{align}
\begin{split}
 &m_k(w_1 \otimes \chi_1, \cdots, w_k \otimes \chi_k)\label{eq:interioraction}\\
 :=&m_k\big(\rho(\chi_2 \cdots \chi_{k})(w_1), \rho(\chi_3 \cdots \chi_{k})(w_2), \cdots, \rho(\chi_k)(w_{k-1}), w_k\big) \otimes \chi_1 \cdots \chi_k.
 \end{split}
\end{align}
\end{definition}
\begin{remark}
This convention makes more sense  after reversing the order of inputs of $m_k$.
This is analogous to the setup in Definition \ref{def:mfa} that to make dg-algebra into an $\AI$-algebra, we take the opposite hom spaces.
\end{remark}
Seidel observed the following isomorphism. For later use, we describe an explicit isomorphism as follows
(taking a sum over $G$-orbit twisted by a character $\chi$).
\begin{lemma}\label{lem:qq}
We have an isomorphism of two $\AI$-algebras
given by $\Phi$. We set $\Phi_{k \geq 2} =0$ and define
$$\Phi_1: hom_{\mathcal{C}^G}(\OL{O}, \OL{O}) \rtimes \WH{G} \to  hom_{\CC}(\WT{O},\WT{O})$$
$$\Phi_1: v \otimes \chi \mapsto \sum_{g \in G} \chi(g^{-1}) (g\cdot v).$$
Furthermore, $\Phi_1$ is $G$-equivariant map where
 $G$-action on the domain of $\Phi_1$ is defined by 
\[ g \cdot (v \otimes \chi) := \chi(g) (v \otimes \chi).\]
\end{lemma}
\begin{proof}
$G$-equivariance is due to the following:
\begin{align*}
\Phi_1(h\cdot(v\otimes \chi))=\Phi_1(\chi(h)v\otimes \chi)=\sum_{g\in G}\chi(g^{-1}h)gv=h\cdot \sum_{g\in G}\chi(g^{-1}h)h^{-1}gv=h\cdot \Phi_1(v\otimes \chi).
\end{align*}
 To prove that it is an $\AI$-morphism, consider the following projection map for $g\in G$.
\[\pi_g: \bigoplus_{g_1,g_2 \in G} hom_{\mathcal{C}}(g_1O, g_2O) \to \bigoplus_{h\in G} hom_{\mathcal{C}}(gO,ghO).\]
Then for $v_i \in hom_{\mathcal{C}}({O},g_i {O})$ and for $g\in G$,
\begin{align*}
 &\pi_g\circ \Big(m_k\big(\Phi_1(v_1\otimes \chi_1),\cdots,\Phi_1(v_k\otimes \chi_k)\big)\Big)\\
 =& m_k\big(\chi_1(g^{-1})gv_1,\chi_2((gg_1)^{-1})gg_1v_2,\cdots,\chi_k((gg_1\cdots g_{k-1})^{-1})gg_1\cdots g_{k-1} v_k\big) \\
 =&( \chi_1\cdots\chi_k)(g^{-1})\cdot (\chi_2\cdots\chi_k)(g_1^{-1})\cdot (\chi_3\cdots\chi_k)(g_2^{-1})\cdots \chi_k(g_{k-1}^{-1})m_k(gv_1,gg_1v_2,\cdots,gg_1\cdots g_{k-1}v_k) \\
 =&( \chi_1\cdots\chi_k)(g^{-1})\cdot (\chi_2\cdots\chi_k)(g_1^{-1})\cdot (\chi_3\cdots\chi_k)(g_2^{-1})\cdots \chi_k(g_{k-1}^{-1})g\cdot m_k(v_1,g_1v_2,\cdots,g_1\cdots g_{k-1}v_k)\\
 =& ( \chi_1\cdots\chi_k)(g^{-1})\cdot g\cdot m_k\big( (\chi_2\cdots\chi_k)(g_1^{-1})v_1,\cdots,\chi_k(g_{k-1}^{-1})v_{k-1},v_k\big) \\
 =& ( \chi_1\cdots\chi_k)(g^{-1})\cdot g\cdot m_k(\rho(\chi_2\cdots\chi_k)(v_1),\cdots,\rho(\chi_k)(v_{k-1}),v_k)\\
 =& \pi_g\circ \big(\Phi_1( m_k(v_1\otimes \chi_1,\cdots,v_k\otimes \chi_k))\big). 
\end{align*}
The map is clearly injective. For $gv \in hom_{\mathcal{C}}(gO,ghO)$, we observe that
\begin{align*}
 \Phi_1\Big(\sum_{\chi\in \HG}\frac{\chi(g)v\otimes \chi}{|\HG|}\Big)
 = \sum_{\chi\in \HG}\sum_{g'\in G}\frac{\chi(gg'^{-1})g'v}{|\HG|}.
 \end{align*}
When $g'=g$, the summand is $gv$. If $g'\neq g$, then the summand is zero due to
 \[ \sum_{\chi\in\HG} \chi(gg'^{-1})=0.\]
 Hence, $\Phi_1$ is also surjective.
\end{proof}
Since $G$ is finite abelian,  $hom_{\CC}(\WT{O},\WT{O})$ has eigenspace decomposition for the $G$-action.
\begin{corollary}
$\Phi_1$ sends  $hom_{\mathcal{C}^G}(\OL{O}, \OL{O})  \otimes \chi$ to a $\chi$-eigenspace 
of  $hom_{\CC}(\WT{O},\WT{O})$.
\end{corollary}

\subsection{Bounding cochains}
We consider Maurer-Cartan theory for the semi-direct product, and the corresponding deformation of $\AI$-structure.
The following observation is easy  but important for further development. It also appeared in Sheridan's work \cite{Sh3}.
\begin{lemma}
If $b$ satisfies weak Mauer-Cartan equation with potential $W$ for the $\AI$-algebra $hom_{\mathcal{C}^G}(\OL{O}, \OL{O})$, then
so does $b \otimes 1$  for the semi-direct product
$hom_{\mathcal{C}^G}(\OL{O}, \OL{O}) \rtimes \WH{G}$.
\end{lemma}
In particular, we can define deformed $\AI$-maps $\{m^{b\otimes 1}\}$ on $hom_{\mathcal{C}^G}(\OL{O}, \OL{O}) \rtimes \WH{G}$.
Note that $m_1^{b\otimes 1}$ preserves $\chi$-eigenspace  $hom_{\mathcal{C}^G}(\OL{O}, \OL{O}) \otimes \chi$ for any $\chi$.
Namely, for $a_1 \otimes \chi$, 
\begin{eqnarray*}
m_1^{b \otimes 1} (a_1 \otimes \chi) &=& \sum_{k=0}^\infty m_{k+1}(b\otimes 1,\cdots, b\otimes 1, a_1\otimes,b\otimes 1,\cdots b\otimes 1) \\
&=& \sum_{k=0}^\infty m_{k+1}(\rho(\chi)(b),\cdots, \rho(\chi)(b), a_1,b,\cdots,b) \otimes \chi 
\end{eqnarray*}
Here $\rho(\chi)(b)$ is a $\chi$ action defined in \eqref{def:cocycleaction}, and therefore acts only on $X_i$'s.
We pretend that $\chi^{-1}$ acts on variables $x_i$'s instead and not on $X_i$'s and make the following definition.
\begin{definition}
For $b=\sum_i x_i X_i$, we set
$$b(\chi^{-1}) := \sum_i \chi^{-1}(x_i) X_i$$
We have $b(\chi^{-1})= \sum_i x_i \chi(X_i)=\rho(\chi)b$.
\end{definition}
Therefore, $m_1^{b \otimes 1} (a_1 \otimes \chi)$ for Floer theory uses contributions of $J$-holomorphic discs with the following inputs
and write the output on the $\chi$-sector.
\begin{figure}[h!]\centering
\begingroup%
  \makeatletter%
  \providecommand\color[2][]{%
    \errmessage{(Inkscape) Color is used for the text in Inkscape, but the package 'color.sty' is not loaded}%
    \renewcommand\color[2][]{}%
  }%
  \providecommand\transparent[1]{%
    \errmessage{(Inkscape) Transparency is used (non-zero) for the text in Inkscape, but the package 'transparent.sty' is not loaded}%
    \renewcommand\transparent[1]{}%
  }%
  \providecommand\rotatebox[2]{#2}%
  \newcommand*\fsize{\dimexpr\f@size pt\relax}%
  \newcommand*\lineheight[1]{\fontsize{\fsize}{#1\fsize}\selectfont}%
  \ifx\svgwidth\undefined%
    \setlength{\unitlength}{170.45109221bp}%
    \ifx\svgscale\undefined%
      \relax%
    \else%
      \setlength{\unitlength}{\unitlength * \real{\svgscale}}%
    \fi%
  \else%
    \setlength{\unitlength}{\svgwidth}%
  \fi%
  \global\let\svgwidth\undefined%
  \global\let\svgscale\undefined%
  \makeatother%
  \begin{picture}(1,0.52495334)%
    \lineheight{1}%
    \setlength\tabcolsep{0pt}%
    \put(0,0){\includegraphics[width=\unitlength,page=1]{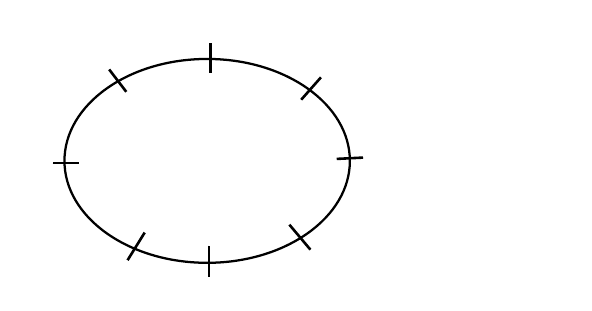}}%
    \put(-0.00444018,0.23423794){\makebox(0,0)[lt]{\lineheight{1.25}\smash{\begin{tabular}[t]{l}$a_1$\\\end{tabular}}}}%
    \put(0.62893833,0.24090514){\makebox(0,0)[lt]{\lineheight{1.25}\smash{\begin{tabular}[t]{l}out\\\end{tabular}}}}%
    \put(0.17112793,0.03644607){\makebox(0,0)[lt]{\lineheight{1.25}\smash{\begin{tabular}[t]{l}b\end{tabular}}}}%
    \put(0.33558407,0.00088804){\makebox(0,0)[lt]{\lineheight{1.25}\smash{\begin{tabular}[t]{l}b\end{tabular}}}}%
    \put(0.53337601,0.05644763){\makebox(0,0)[lt]{\lineheight{1.25}\smash{\begin{tabular}[t]{l}b\end{tabular}}}}%
    \put(0.07112079,0.43869704){\makebox(0,0)[lt]{\lineheight{1.25}\smash{\begin{tabular}[t]{l}$b(\chi^{-1})$\end{tabular}}}}%
    \put(0.28454974,0.47980669){\makebox(0,0)[lt]{\lineheight{1.25}\smash{\begin{tabular}[t]{l}$b(\chi^{-1})$\end{tabular}}}}%
    \put(0.54234589,0.41535772){\makebox(0,0)[lt]{\lineheight{1.25}\smash{\begin{tabular}[t]{l}$b(\chi^{-1})$\end{tabular}}}}%
  \end{picture}%
\endgroup%

\caption{Geometric description of $m_1^{b(\chi^{-1}),b}(a_1\otimes \chi)$}
\label{chi}
\end{figure}
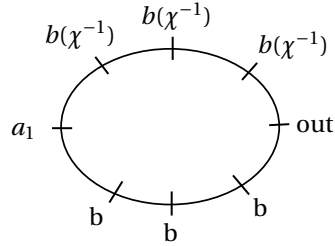
\begin{lemma} \label{lem:twisteddiff}
We have an isomorphism sending $w\otimes \chi \to w$.
$$\big(hom_{\mathcal{C}^G}(\OL{O}, \OL{O}) \otimes \chi,  m_1^{b\otimes 1} \big) \cong \big(hom_{\mathcal{C}^G}(\OL{O}, \OL{O}), m_1^{b(\chi^{-1}),b} \big).$$
\end{lemma}
In general,  $m_k^{b\otimes 1}$ is given by
\begin{align}
& m_k^{b\otimes 1}(a_1\otimes \chi_1,\cdots,a_k\otimes \chi_k)\label{twistedainfty}\\
=&m_k^{\rho(\chi_1\cdots\chi_k)b,\rho(\chi_2\cdots\chi_k)b,\cdots,\rho(\chi_k)b,b}\big(\rho(\chi_2\cdots\chi_k)a_1,\rho(\chi_3\cdots\chi_k)a_2,\cdots,a_k\big)\otimes \chi_1\cdots\chi_k.\nonumber
\end{align}
In particular, $m_2^b$ product of $\chi_1$ and $\chi_2$-eigenvectors goes to $\chi_1\chi_2$-eigenspace.

Sometimes it is convenient to work on  $hom_{\mathcal{C}}(\WT{O},\WT{O})$.
\begin{definition}\label{def:btil}
For a bounding cochain $b$, we set
$$\WT{b} := (\Phi)_*(b) = \Phi_1(b)$$
In particular, for $b({x}) = \sum_i x_i X_i \otimes 1$, 
we get 
$$\WT{b}({x}) = \sum_i x_i  \big( \sum_{g \in G} g(X_i) \big)$$
\end{definition}
As we assumed that $g_1O \neq g_2O$ for $g_1 \neq g_2$, $G$-action permutes the
output of $m_k$-operation  in $\hom(\WT{O},\WT{O})$ and  one can observe that
$W(\WT{b}) = W(b)$.

We remark that in \cite{CHL}, the following localized mirror functor has been defined.
\begin{theorem}[\cite{CHL}]\label{thm:upstairfunctor}
We have an  $\AI$-functor ${\mathcal{F}}^{\WT{O}}$ (resp. $\mathcal{F}^{\OL{O}}$) which are cohomologically injective on $\Hom$'s with $\WT{O}$ (resp. $\OL{O}$).
\[\xymatrix{
\mathcal{C} \ar[rr]^{\WT{\mathcal{F}}^{\WT{O}}} \ar[d]_{\textrm{quotient by} \; G} && MF^\AI_\HG(W)  \\
\mathcal{C}^G \ar[rr]^{\mathcal{F}^{\OL{O}}} && MF^\AI(W)\ar[u]_{\textrm{quotient by} \; \WH{G} }}\]

\end{theorem}
\subsection{ $\AI$-algebra $\CB\rtimes \WH{G}$ and its $\WH{G}$-quotient}

Recall that in non-equivariant case, we  defined a new $\AI$-algebra $\mathcal{B}$ by tensoring $R = \Lambda[x_1,\cdots, x_n]$ to an $\AI$-algebra $\mathcal{A}$.
\begin{definition}
An $\AI$-algebra $\CB\rtimes \WH{G}$ is the data
$$\big((hom_{\mathcal{C}^G}(\OL{O}, \OL{O}) \rtimes \HG) \otimes R,\{m_k^{b\otimes 1}\}\big)$$
where $m_k$ is the $R$-linear extension of $m_k$ on $hom_{\mathcal{C}^G}(\OL{O}, \OL{O}) \rtimes \HG$, and $b\otimes 1 =\sum x_i X_i \otimes 1$ is a Maurer-Cartan element for $\CB\rtimes \HG$.
\end{definition}

\begin{lemma}\label{lem:isobr}
Using  Lemma \ref{lem:qq} (tensoring $R$), we can identify the following two $\AI$-algebras via  $\AI$-isomorphism $\Phi$.
\[ \big( \CB \rtimes \HG, \{m_k^{b\otimes 1}\}\big) \simeq \big( \hom_{\mathcal{C}}(\WT{O},\WT{O})\otimes R, \{m_k^{\WT{b}}\}\big).\]
\end{lemma}

We now define the second $\WH{G}$-action on $\CB\rtimes \HG$ (which is strict).
\begin{definition}\label{def:hatac}
First, for $X_i \in hom_{\mathcal{C}}(O, g_i O)$, recall from \eqref{def:cocycleaction}
\[\chi \cdot X_i = \rho(\chi)(X_i)=\chi(g_i^{-1}) X_i.\]
For the dual variable, we set $$ \chi \cdot x_i := \chi(g_i) x_i.$$
Then we define the $\HG$-action on $\CB \rtimes \HG$ by
\[ \chi\cdot(r(x)v\otimes \eta):= r(\chi\cdot x) \rho(\chi)v\otimes \eta.\]
\end{definition}
\begin{remark}
On $\rtimes \HG$ part, $\HG$ is supposed to act by conjugation in previous literatures, but since $\HG$ is abelian, our action on $\HG$-part is trivial.
\end{remark}
For this second $\WH{G}$ action, we have $\chi \cdot (b\otimes 1) = (b \otimes 1)$, where $\chi$-action on $x_i$ and $X_i$ cancels each other.
\begin{remark}
Two $\HG$-actions  in Definition \ref{def:hatac0} and  in Definition \ref{def:hatac} are different.
The first $\HG$-action in Definition \ref{def:hatac0} only acts on $hom(\OL{O},\OL{O})$ and is used to define semi-direct product $\mathcal{B}\rtimes \WH{G}$. The second  actions in Definition \ref{def:hatac} is an action on $\mathcal{B}\rtimes \WH{G}$, and we will mostly use the second action from now on.
\end{remark}

\begin{prop}\label{ghatequivariance}
The $\HG$-action on $\CB\rtimes \HG$ is compatible with the $\AI$-structure $\{m_k^{b\otimes 1}\}$.
\end{prop}

\begin{proof}
Let $f_i v_i\otimes \chi_i \in \CB\rtimes \HG$ for $i=1,\cdots,k$, where $f_i \in R$ and $v_i \in hom_{\mathcal{C}}({O},g_i{O})$. Let $\eta \in \HG$. Then,
\begin{align*}
 m_k(\eta\cdot(f_1v_1\otimes \chi_1),\cdots,\eta\cdot(f_kv_k\otimes\chi_k))&= (\eta\cdot(f_1\cdots f_k)) m_k(\eta(g_1^{-1})v_1\otimes\chi_1,\cdots,\eta(g_k^{-1})v_k\otimes \chi_k)\\
&=( \eta \cdot(f_1\cdots f_k))\eta(g_1^{-1}\cdots g_k^{-1})m_k(v_1\otimes\chi_1,\cdots,v_k\otimes\chi_k)\\
&= (\eta\cdot(f_1\cdots f_k))\eta\cdot m_k(v_1\otimes\chi_1,\cdots,v_k\otimes\chi_k)\\
&= \eta\cdot m_k(f_1v_1\otimes\chi_1,\cdots,f_kv_k\otimes \chi_k).
\end{align*}
The third equality comes from the fact 
\[ m_k\big((v_1)_{g_1},\cdots,(v_k)_{g_k}\big) \in hom_{\mathcal{C}}({O},g_1\cdots g_k{O})\]
and the definition of $\HG$-action \eqref{def:cocycleaction}.
Since $b \otimes 1$ is invariant under $\HG$-action, $m_k^{b\otimes 1}$ is also compatible with the action.
\end{proof}

\begin{corollary}
$\WH{G}$-invariant part  of $\CB\rtimes \HG$, denoted as $(\CB\rtimes \HG)^{\WH{G}}$,  has an induced $\AI$-structure $\{m_k^{b\otimes 1}\}$.
\end{corollary}

\section{Cohomology algebra of $\mathcal{B} \rtimes \WH{G}$ and orbifold Jacobian ring}\label{sec:eqbj}
From an $\AI$-algebra $hom_{\mathcal{C}^G}( \OL{O}, \OL{O})$ with potential function $W$, we constructed an $\AI$-algebra $\mathcal{B}\rtimes \WH{G}$ and its $\WH{G}$-quotient  $(\mathcal{B}\rtimes \WH{G})^{\WH{G}}$ in the previous section. We find  the relation to the orbifold Jacobian ring of $(W, \WH{G})$. 
\begin{theorem}\label{thm:main}
Assume that $W$ has an isolated singularity at the origin.
Suppose an $\AI$-algebra  $hom_{\mathcal{C}^G}( \OL{O}, \OL{O})$ satisfies an Assumption \ref{as2}. 
Then we have an algebra isomorphism:
$$  H^*\big( (\mathcal{B}\rtimes \WH{G})^{\WH{G}}\big)_{alg} \cong \Jac(W,\WH{G}).$$
\end{theorem}
\begin{proof}[Proof of the Theorem]
We will use the fact that
$$\Jac(W,\WH{G}) \cong \Hom_{MF_{\WH{G} \times \WH{G}} (W(y) - W(x))}( \Delta^{\WH{G} \times \WH{G}}_W, \Delta^{\WH{G} \times \WH{G}}_W).$$
Let us briefly explain this.
By Shklyarov, $\Jac(W,\WH{G})$ is isomorphic to the Hochschild cohomology $H^*(MF_{\WH{G}}(W), MF_{\WH{G}}(W))$ which is (from Definition \ref{def:hhcohomology}) 
$$hom_{R\underline{Hom}_c(MF_{\WH{G}}(W), MF_{\WH{G}}(W))}(\mathrm{Id}_{MF_{\WH{G}}(W)},\mathrm{Id}_{MF_{\WH{G}}(W)}).$$
 On the other hand, Polishchuk-Vaintrob \cite{PV} showed that
$$R\underline{Hom}_c (MF_{\WH{G}}(W), MF_{\WH{G}}(W)) \cong MF_{\WH{G} \times \WH{G}} (W(y) - W(x))$$
such that identity functor $\mathrm{Id}_{MF_{\WH{G}}(W)}$ corresponds to a kernel $\Delta^{\WH{G} \times \WH{G}}_W$.

To prove the main theorem, we will show how to relate $\Delta^{\WH{G} \times \WH{G}}_W$ and $\mathcal{B} \rtimes \WH{G}$ using localized mirror functor of \cite{CHL}. Let us make the following shorthand notation.
\[\OL{O}^{x} := \big(\OL{O},b(x) \big), \WT{O}^{y} := \big(\WT{O}, \WT{b}(y)\big).\]
Extend the $\HG$-action on $R^e$ by \[ \chi\cdot y_i:=\chi(g_i)y_i\] when $X_i \in \Hom(O,g_i O)$. 
\begin{lemma}\label{prop:eo}
With respect to the $\HG$-action on $hom_{\mathcal{C}^G}(\OL{O},\OL{O})$ and $R^e$,
\begin{equation}\label{eq:m1}
hom_{\mathcal{C}^G}(\OL{O}^x,\OL{O}^y)\otimes R^e,m_1^{b(x),b(y)})
\end{equation}
 is a $\HG$-equivariant matrix factorization of $W(y)-W(x)$.
\end{lemma}
\begin{proof}[Proof of the Lemma]
Let $v\in hom_{\mathcal{C}^G}(\OL{O}^x,\OL{O}^y)\otimes R^e.$ By Proposition \ref{ghatequivariance}, 
\[\chi\cdot m_1^{b(x),b(y)}(v)=m_1^{\chi\cdot b(x),\chi\cdot b(y)}(\chi\cdot v).\]
Since $\chi\cdot b(x)=b(x)$ and $\chi\cdot b(y)=b(y)$,  $m_1^{b(x),b(y)}$ is $\HG$-equivariant.
\end{proof}
From  Assumption \ref{as2}  and by Theorem \ref{PVtheorem} (and appealing to Remark \ref{rmk:equivkernel}), we can
obtain $\Delta_W^{\HG\times\HG}$ from the matrix factorization \eqref{eq:m1}.
\[\Delta_W^{\HG\times\HG}:= \bigoplus_{\chi\in \HG} \big(hom_{\mathcal{C}^G}(\OL{O}^{\chi\cdot x},\OL{O}^y)\otimes R^e, m_1^{b(\chi\cdot x),b(y)}\big).\]
Using Lemma \ref{lem:twisteddiff}, we rewrite the kernel as
\[ \Delta_W^{\HG\times\HG}=\bigoplus_{\chi\in \HG} \big((hom_{\mathcal{C}^G}(\OL{O}^x,\OL{O}^y)\otimes \chi)\otimes R^e, m_1^{b(x)\otimes 1,b(y)\otimes 1}\big).\]
Then the $\HG\times\HG$-action on $\Delta_W^{\HG\times\HG}$ in \eqref{eq:hhaction} translates into
\[ (\chi_1\times \chi_2)\cdot (r(x,y)v\otimes \chi)= r(\chi_1\cdot x,\chi_2\cdot y)\rho(\chi_2)(v)\otimes \chi_2\chi\chi_1^{-1}.\]
Observe that if we restrict the action to the diagonal subgroup $\HG$, then it coincides with the action in Definition \ref{def:hatac}.

The following two propositions will prove the main theorem.
\begin{prop}\label{prop:m1}
There is an $\AI$-homomorphism between two $\AI$-algebras
\begin{equation}\label{comparingmap}
\mathcal{F}: \mathcal{B} \rtimes \WH{G} \to hom_{MF^{\AI}(W(y) - W(x))}( \Delta^{\WH{G} \times \WH{G}}_W,  \Delta^{\WH{G} \times \WH{G}}_W).
\end{equation}
Moreover, $\mathcal{F}_1$ is injective in cohomology.
\end{prop}

\begin{prop}\label{prop:m2}
The cohomological image of $\mathcal{B} \rtimes \WH{G}$  under $\mathcal{F}_1$ are exactly morphisms 
that are $(1 \times \HG)$-equivariant, hence given by
\begin{equation}\label{eq:1ghat}
\Hom_{MF^{\AI}_{1\times \HG}(W(y) - W(x))}( \Delta^{\WH{G} \times \WH{G}}_W,  \Delta^{\WH{G} \times \WH{G}}_W).
\end{equation}
The cohomological image of $(\mathcal{B}\rtimes\HG)^\HG$ are exactly  $\HG \times \HG$-equivariant morphisms given by
$$\Hom_{MF^{\AI}_{\WH{G}\times \WH{G}}(W(y) - W(x))}( \Delta^{\WH{G} \times \WH{G}}_W,  \Delta^{\WH{G} \times \WH{G}}_W).$$
\end{prop}

\begin{proof}[Proof of Proposition \ref{prop:m1}]
We use the idea of localized mirror functor in \cite{CHL} to define an $\AI$-functor $\mathcal{F}$ to matrix factorizations.
Recall from \cite{CHL} that  given an $\AI$-algebra $\hom(\WT{O},\WT{O})$ with bounding cochains $\WT{b}(y)$ and potential function $W(y)$, an $\AI$-functor 
$$\mathcal{F}^{\WT{O}^y}: \mathcal{C} \to MF(W(y))$$
(relative to $(\WT{O},\WT{b}(y))$) is defined
by sending an object $K$ of $\mathcal{C}$ to the matrix factorization $\big( \hom(K, \WT{O}), -m_1^{0,\WT{b}(y)}\big)$.
Higher part of the functor is defined as
$$\mathcal{F}^{\WT{O}^y}_k(p_1,\cdots,p_k) = m_{k+1}(p_1,\cdots,p_k,\cdot).$$
Lemma 7.19 of \cite{CHL}  states that this functor is cohomologically injective. This was shown by 
constructing an explicit right inverse using the unit of $\AI$-algebra.
We remark that we use the sign convention of \cite{CHLnc} by taking $\hom( \cdot, \WT{O})$ instead of $\hom(\WT{O},\cdot)$.

For the proof, we will use the following variation of the above construction. 
Namely, we can apply the functor $\mathcal{F}^{\WT{O}^y}$ to the  same object $\WT{O}$ but equipped with a bounding cochain $\WT{b}(x)$.
In this case, $(\WT{O},\WT{b}(x))$ is mapped by $\mathcal{F}^{\WT{O}^y}$ to
$$\big( hom_{\mathcal{C}}(\WT{O}^x,\WT{O}^y)\otimes R^e,-m_1^{\WT{b}(x),\WT{b}(y)} \big),$$
which is a  matrix factorization of
$W(y) - W(x)$.
Also, $k$-th part of the $\AI$-functor in this case is given as follows.
\begin{align*}
 ({\CF}^{\WT{O}^y})_k:&\;\; \big(hom_{\mathcal{C}}(\WT{O}^x,\WT{O}^x)\otimes R^e \big)^{\otimes k}\\
 \to & hom_{MF^{\AI}(W(y) -W(x))}\big((hom_{\mathcal{C}}(\WT{O}^{x},\WT{O}^{y})\otimes R^e,-m_1^{\WT{b}(x),\WT{b}(y)}),(hom_{\mathcal{C}}(\WT{O}^{x},\WT{O}^{y})\otimes R^e,-m_1^{\WT{b}(x),\WT{b}(y)})\big), \end{align*}
\[ (p_1,\cdots,p_k)  \mapsto m_{k+1}^{\WT{b}(x),\cdots,\WT{b}(x),\WT{b}(y)}(p_1,\cdots,p_k,\bullet).\]

We can show that this defines an $\AI$-homomorphim by the same argument in \cite{CHL} (hence omit the proof). The injectivity can be also shown as in \cite{CHL}.


We modify the sign for matrix factorization category using the following simple lemma
\begin{lemma}\label{signchainmap}
For an $\AI$-algebra $(A,m_1,m_2, m_{\geq 3}=0)$, there is an $\AI$-isomorphism
\[ (A,m_1,m_2) \simeq (A,-m_1,m_2),\;\; a \mapsto (-1)^{|a|}a.\] 
\end{lemma}
Applying this to the object $(hom_{\mathcal{C}}(\WT{O}^{x},\WT{O}^{y})\otimes R^e,-m_1^{\WT{b}(x),\WT{b}(y)})$,
we remove the negative sign.
Using the isomorphism in Lemma \ref{lem:isobr}, we have
\[ \Delta_W^{\HG\times\HG}=\big((hom_{\mathcal{C}^G}(\OL{O}^x,\OL{O}^y)\rtimes \HG)\otimes R^e, m_1^{b(x)\otimes 1,b(y)\otimes 1}\big) \simeq \big(hom_{\mathcal{C}}(\WT{O}^x,\WT{O}^y)\otimes R^e,m_1^{\WT{b}(x),\WT{b}(y)}\big).\]
Combining these isomorphisms, we get the desired $\AI$-morphism
\[ \CF:\mathcal{B} \rtimes \WH{G} \to hom_{MF^{\AI}(W(y) - W(x))}( \Delta^{\WH{G} \times \WH{G}}_W,  \Delta^{\WH{G} \times \WH{G}}_W). \qedhere\]
\end{proof}

Now, it remains to verify the assertions about equivariance of morphisms.
\begin{proof}[Proof of Proposition \ref{prop:m2}]
We first prove that the image of $\CF_1$ is included in 
\[hom_{MF^\AI_{1\times\HG}(W(y)-W(x))}(\Delta_W^{\HG\times\HG},\Delta_W^{\HG\times\HG}).\] 
Let 
\[f(x)p\otimes\eta \in (hom_{\mathcal{C}^G}(\OL{O}^x,\OL{O}^x)\otimes \eta)\otimes R,\] 
and 
\[r(x,y)v\otimes \chi \in (hom_{\mathcal{C}^G}(\OL{O}^x,\OL{O}^y)\otimes \chi)\otimes R^e.\]
By definition of $\HG\times\HG$-action,
\begin{align}
&(\chi_1\times \chi_2)\cdot \big(m_2^{b(x)\otimes 1,b(x)\otimes 1,b(y)\otimes 1}(f(x)p\otimes\eta, r(x,y)v\otimes \chi)\big) \label{eq:hhequiv1}\\
=& \rho(\chi_2)\Big(m_2^{\rho(\eta\chi)b(\chi_1\cdot x),\rho(\chi)b(\chi_1\cdot x),b(\chi_2\cdot y)}\big(\rho(\chi)(f(\chi_1\cdot x)p),r(\chi_1\cdot x,\chi_2\cdot y)v\big)\Big)\otimes\chi_2\eta\chi\chi_1^{-1}\nonumber\\
=& m_2^{\rho(\chi_2\eta\chi)b(\chi_1\cdot x),\rho(\chi_2\chi)b(\chi_1\cdot x),\rho(\chi_2)b(\chi_2\cdot y)}\big(\rho(\chi_2\chi)(f(\chi_1\cdot x)p),\rho(\chi_2)(r(\chi_1\cdot x,\chi_2\cdot y)v)\big)\otimes \chi_2\eta\chi\chi_1^{-1}\nonumber\\
=& m_2^{\rho(\chi_2\eta\chi\chi_1^{-1})b(x),\rho(\chi_2\chi\chi_1^{-1})b(x),b(y)}\big(\rho(\chi_2\chi)(f(\chi_1\cdot x)p),\rho(\chi_2)(r(\chi_1\cdot x,\chi_2\cdot y)v)\big)\otimes\chi_2\eta\chi\chi_1^{-1}.\nonumber
\end{align}
On the other hand,
\begin{align}
&  m_2^{b(x)\otimes 1,b(x)\otimes 1,b(y)\otimes 1}\big(f(x)p\otimes\eta, (\chi_1\times\chi_2)\cdot(r(x,y)v\otimes \chi)\big)\label{eq:hhequiv2}\\
=& m_2^{b(x)\otimes 1,b(x)\otimes 1,b(y)\otimes 1}\big(f(x)p\otimes\eta, r(\chi_1\cdot x,\chi_2\cdot y)\rho(\chi_2)v\otimes \chi_2\chi\chi_1^{-1}\big)\nonumber\\
=& m_2^{\rho(\eta\chi_2\chi\chi_1^{-1})b(x),\rho(\chi_2\chi\chi_1^{-1})b(x),b(y)}\big(\rho(\chi_2\chi\chi_1^{-1})(f(x)p),r(\chi_1\cdot x,\chi_2\cdot y)\rho(\chi_2)v\big)\otimes \eta\chi_2\chi\chi_1^{-1}.\nonumber
\end{align}
If $\chi_1=1$, then \eqref{eq:hhequiv1}=\eqref{eq:hhequiv2}, so the morphism 
\[\CF_1(f(x)p\otimes\eta)=m_2^{b(x)\otimes1,b(x)\otimes1,b(y)\otimes1}(f(x)p\otimes\eta,\bullet)\] 
is $1\times\HG$-equivariant for any $f(x)p\otimes \eta$.

$\CF_1(f(x)p)\otimes\eta$ is a $\HG\times\HG$-equivariant morphism if and only if \eqref{eq:hhequiv1}=\eqref{eq:hhequiv2} for general $\chi_1 \times \chi_2$. It is equivalent to
\[f(\chi_1\cdot x)p= \rho(\chi_1^{-1})(f(x)p),\]
if and only if $f(x)p$ is $\HG$-invariant.

Since we have an injectivity result by Proposition \ref{prop:eo}, it suffices to show that any cohomology class of a $1\times \HG$-invariant morphism is an image of $\CF_1$. Then it proves not only the first but also the second statement, because any $\HG\times \HG$-equivariant morphism is a $1\times \HG$-equivariant morphism. For this, we show the following isomorphism of chain complexes
\[ (hom_{\mathcal{C}^G}(\OL{O}, \OL{O}) \rtimes \HG) \otimes R \simeq hom_{MF^{\AI}_{1\times\HG}(W(y)-W(x))}(\Delta_W^{\HG\times\HG},\Delta_W^{\HG\times\HG}),\]
which implies that
\begin{equation*}\label{moduleiso} H^*(\CB \rtimes \HG) \cong \Hom_{MF^{\AI}_{1\times \HG}(W(y)-W(x))}(\Delta_W^{\HG\times\HG},\Delta_W^{\HG\times\HG}).\end{equation*}
Let $\phi\in hom_{MF^{\AI}_{1\times\HG}(W(y)-W(x))}(\Delta_W^{\HG\times\HG},\Delta_W^{\HG\times\HG})$. Then the equivariance of $\phi$ gives us
\[ (1\times\chi^{-1})\cdot \phi(rv\otimes\chi)=\phi\big((1\times\chi^{-1})\cdot(rv\otimes\chi)\big)=\phi\big(r(x,\chi^{-1}\cdot y)\rho(\chi^{-1})v\otimes 1\big).\]
We conclude that a $1\times\HG$-equivariant morphism $\phi$ is completely determined by its restriction to $hom(\OL{O}^x,\OL{O}^y)\otimes 1$. Therefore,
\begin{align*}
& hom_{MF^{\AI}_{1\times\HG}(W(y)-W(x))}(\Delta_W^{\HG\times\HG},\Delta_W^{\HG\times\HG})\\
\simeq& hom_{R^e}\big((hom_{\mathcal{C}^G}(\OL{O}^x,\OL{O}^y)\otimes 1)\otimes R^e,\bigoplus_{\chi\in \HG} (hom_{\mathcal{C}^G}(\OL{O}^x,\OL{O}^y)\otimes \chi)\otimes R^e\big)\\
\simeq& hom_{R^e}\big(hom_{\mathcal{C}^G}(\OL{O}^x,\OL{O}^y)\otimes R^e,\bigoplus_{\chi\in\HG} hom_{\mathcal{C}^G}(\OL{O}^{\chi\cdot x},\OL{O}^y)\otimes R^e\big)\\
\simeq& hom_{R^e}\big( hom_{\mathcal{C}^G}(\OL{O}^x,\OL{O}^y)\otimes R^e,\bigoplus_{\chi\in\HG}R^e/(y- \chi\cdot x)  \big) \\
\simeq& \bigoplus_{\chi\in\HG} (hom_{\mathcal{C}^G}(\OL{O}^x,\OL{O}^{\chi\cdot x})\otimes R)^\vee.
\end{align*}
By Assumption \ref{as2}, the differential of $hom_{\mathcal{C}^G}(\OL{O}^x,\OL{O}^{\chi\cdot x})\otimes R$ is a sum of Koszul differentials. By self-duality of Koszul complexes, we have
\[ (hom_{\mathcal{C}^G}(\OL{O}^x,\OL{O}^{\chi\cdot x})\otimes R)^\vee \simeq \bigoplus_{\chi\in\HG} hom_{\mathcal{C}^G}(\OL{O}^{ x},\OL{O}^{\chi\cdot x})\otimes R.\]
On each $\chi$-summand, we take coordinate change $x \mapsto \chi^{-1}\cdot x$. Then we have
\[hom_{\mathcal{C}^G}(\OL{O}^x,\OL{O}^{\chi\cdot x})\otimes R\simeq hom_{\mathcal{C}^G}(\OL{O}^{\chi^{-1}\cdot x},\OL{O}^{x})\otimes R,\]
hence
\[hom_{MF^{\AI}_{1\times\HG}(W(y)-W(x))}(\Delta_W^{\HG\times\HG},\Delta_W^{\HG\times\HG})\simeq \bigoplus_{\chi\in \HG} hom_{\mathcal{C}^G}(\OL{O}^{\chi^{-1}\cdot x},\OL{O}^x)\otimes R \simeq (hom_{\mathcal{C}^G}(\OL{O},\OL{O})\rtimes \HG)\otimes R, \]
and we finish the proof of the Proposition.
\end{proof}
Let us finish the proof of the Theorem \ref{thm:main}. By Proposition \ref{prop:m2}, we have
\[ H^*((\CB\rtimes \HG)^\HG) \cong \Hom_{MF_{\HG\times\HG}^\AI(W(y)-W(x))}(\Delta_W^{\HG\times\HG},\Delta_W^{\HG\times\HG}).\]
Modifying sign to the dg-setting,
\[ H^*((\CB\rtimes \HG)^\HG)_{alg} \cong \Hom_{MF_{\HG\times\HG}(W(y)-W(x))}(\Delta_W^{\HG\times\HG},\Delta_W^{\HG\times\HG})\cong \Jac(W,\HG). \qedhere\] 

\end{proof}

It is natural to conjecture that the theorem holds before taking $\WH{G}$-quotient. 
\begin{conjecture} With the same assumptions as in Theorem \ref{thm:main}, we have
\[ H^* (\mathcal{B}\rtimes \WH{G})_{alg} \cong \Jac'(W,\HG).\]
\end{conjecture}
Although we have proved that left hand side equals  \eqref{eq:1ghat}, we do not know whether this is isomorphic to twisted Jacobian algebra defined by \cite{BTW,Shk}.

\section{Kodaira-Spencer map}\label{sec:KS}
In this section, we construct a general Kodaira-Spencer map from small quantum cohomology $QH^*(M)$ to a cohomology
of an $\AI$-algebra, which is isomorphic to Jacobian ring of $W$ for a non-equivariant case and
orbifold Jacobian algebra of $(W,\WH{G})$ for an equivariant case.
The construction should extend to the case of big quantum cohomology using bulk-deformations by following \cite{FOOO_MS}
but we omit it for simplicity.

Recall that Fukaya-Oh-Ohta-Ono\cite{FOOO_MS} defined a geometric map $\mathfrak{ks}:QH^*(M) \to \Jac(W)$ for general toric manifolds
(with bulk-deformations) and showed that the map is a ring isomorphism. 
Such a construction was generalized to the case of orbi-sphere $\mathbb{P}^1_{a,b,c}$ \cite{ACHL}.
The former construction uses $T^n$-action in an essential way. For example, for the Fukaya algebra $\mathcal{A}(L)$ of a Lagrangian torus $L$, it uses the observation that the output of the closed-open map (using
a $J$-holomorphic disc with one interior input and one boundary output) $$\mathcal{CO}_0: QH^*(M) \to H^*(\mathcal{A}(L))_{alg}$$
is always a multiple of the unit $\be = [L]$ from $T^n$-action. In this case, we can just read the coefficient of the unit
with boundary Maurer-Cartan deformation and this gives an element of the algebra $\Jac(W)$.  The construction of \cite{ACHL} is similar, but uses  $\Z/2$-action (and low dimensionality of $L$) instead of $T^n$-action and show that $\mathcal{CO}_0$ always maps to the multiple of $\be$.

But in general, the image of the map $\mathcal{CO}_0$ is not expected to be a multiple of $\be$. For example, its virtual dimension is
$n + \mu(\beta)-\deg_M(A)$ which is not necessarily $n = \dim(L)$.

Our simple but important idea is that we replace $\Jac(W)$ by its associated Koszul complex. Also, we find
a mild assumption \ref{as1} with which the $\AI$-algebra $(\mathcal{B}, m_1^b)$ from $L$ defines such a Koszul complex.
In fact, $\mathcal{CO}_0$ naturally has an output in $\mathcal{B}$ (see Figure \ref{disk} (a)) by decorating $L$ with bounding cochain $b$.
Here, we work in the setting of localized mirror in the sense that we only look at the part of the mirror given by the formal neighborhood of
the reference Lagrangian $\bL$.
\begin{theorem}\label{thm:ks}
Let $\mathcal{A}(L)$ be Fukaya $\AI$-algebra of Lagrangian $L$, with bounding cochain $b$
and potential function $W$. Denote by $\mathcal{B}(L)$ another $\AI$-algebra defined in Definition \ref{def:B}. 
There is a Kodaira-Spencer map which is a ring homomorphism
$$\mathfrak{ks}:QH^*(M,\Lambda) \to H^*(\mathcal{B}(L))_{alg}.$$
In particular    for $f_1,f_2 \in QH^*(M)$, we have
$$\ks(f_1 \ast_Q f_2) = (-1)^{\deg \ks(f_1)} m_2(\ks(f_1), \ks(f_2)).$$
Under the Assumption \ref{as1}, the target is isomorphic to $\Jac(W_L)$ by Proposition \ref{prop:bjac}.
\end{theorem}
\begin{figure}
\includegraphics[height=1.3in]{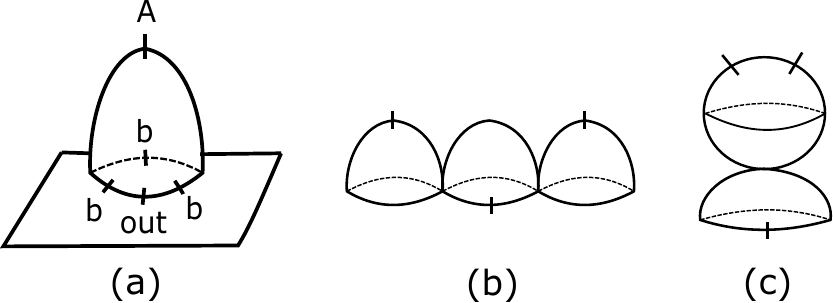}
\caption{(a) Kodaira-Spencer map $\ks(A)$, (b) $\Sigma_{12}$, (c) $\Sigma_0$}
\label{disk}
\end{figure}
Most of the construction of Fukaya-Oh-Ohta-Ono\cite{FOOO_MS}, and its modifications \cite{ACHL}  carries over to this setting
and we give a sketch of proof later in the section (explaining which parts of \cite{FOOO_MS} and \cite{ACHL} have to be modified).

Let us explain an equivariant version of  Kodaira-Spencer map. Let $G$ be a finite abelian group and
 $M$ be a symplectic manifold with an effective $G$-action. Fukaya category of $M$ can be given as a filtered $\Z/2$-graded unital $\AI$-category
 with a strict $G$-action (using de Rham version of the work of Fukaya-Oh-Ohta-Ono but any other technical setting will be okay).
 Therefore, we can apply the construction of Section 5.
Suppose the Lagrangian $\OL{L} \subset [M/G]$ satisfies the assumptions of Section 5 so that
it has an $\AI$-algebra $\mathcal{A}(\OL{L})$ as well as the $\AI$-category $\WT{\mathcal{A}}$ whose
objects are $|G|$ embedded lifts of $\OL{L}$. $b = \sum x_i X_i$ satisfies a Maurer-Cartan equation for $\mathcal{A}(\OL{L})$.
\begin{theorem}\label{thm:ks2}
There is a  Kodaira-Spencer map which is a ring homomorphism
\begin{equation}\label{eq:ksb}
\mathfrak{ks}:QH^*(M,\Lambda) \to H^*\big(( \mathcal{B}(\OL{L}) \rtimes \WH{G})^{\WH{G}}\big)_{alg}.
\end{equation}
Under the Assumption \ref{as2}, the target of $\ks$ is isomorphic to the orbifold Jacobian algebra $\Jac(W_{\OL{L}}, \WH{G})$ by Theorem \ref{thm:main}.
\end{theorem}
\begin{remark}
If we restrict the above $\ks$ map to the $G$-invariant part ($\chi =1$ case), it may be regarded as a map from untwisted sector of the quotient orbifold $H^*([M/G])$ to  $H^*\big( \mathcal{B}(\OL{L})^{\WH{G}} , m_1^b\big)$. 
This should be a part of orbifold version of
$$\mathfrak{ks}:QH^*_{orb}([M/G],\Lambda) \to H^*( \mathcal{B}(\OL{L}) )_{alg}.$$
 We can define the map when the input is a fundamental class of a twisted sector following \cite{ACHL}, but there is a technical difficulty to define it in general cases. We will not consider this map in this paper.
\end{remark}
Now, let us prove Theorem \ref{thm:ks} and Theorem \ref{thm:ks2}.

\subsection{Proof of Theorem \ref{thm:ks}}
We will mainly follow the construction of \cite{FOOO_MS}, \cite{ACHL}.
We construct a Kodaira-Spencer map using $J$-holomorphic discs as in Figure \ref{disk}.
As explained before, the main difference in our case is that we do not need to make an output
a multiple of the fundamental class. This is why we can consider a general case. $T^n$-action of \cite{FOOO_MS} or $\Z/2$-symmetry
of \cite{ACHL} was used to achieve this. But we still need the Lagrangian $L$ to be weakly unobstructed.
After this adjustment, it is almost straightforward to adapt the proof of \cite{FOOO_MS}, \cite{ACHL} to our setup, and
we only give a brief sketch, and refer readers to these references for the full construction. 
%
Also we are considering a general case where $QH^*(M)$ may have odd degree elements (like $T^2$), our formula involves
a sign. We explain the sign computation at the end.

For a moduli space of $J$-holomorphic discs $\mathcal{M}_{k+1,1}(\beta)$ for $\beta \in H_2(M,L)$ and a cycle $A \subset M$,
we define a Kodaira-Spencer map as follows.
Let $\mathcal{M}_{k+1,1}(\beta,A) = \mathcal{M}_{k+1,1}(\beta) \times_M A$ and consider their evaluation maps $ev_i^k:\mathcal{M}_{k+1,1}(\beta,A) \to L$ at $i$-th marked point.
We define
\begin{equation}\label{eq:defks}
\ks (PD[A]) = \sum_{k=0}^\infty (ev_0^k)_! ((ev_1^k)^*b \wedge \cdots \wedge (ev_k^k)^*b)
\end{equation}

From now on, we omit Poincar\'e dual from  $PD[A]$ and just write $A$.
Gromov compactness shows that $\ks(A) \in \mathcal{B}$.
We first check that this map is well-defined in cohomology.
\begin{lemma}[c.f. Proposition 2.4.15 \cite{FOOO_MS}]\label{lem:ksc}
If $A -B = \partial R$, then $\ks(A) - \ks(B) = m_{1,can}^b H$.
\end{lemma}
\begin{proof}
We define $H:=\ks(R)$, and consider its codimension one boundary contributions.
One of the strata is $\ks(\partial R)$ which corresponds to $\ks(A) - \ks(B)$. The others come from disc bubbling.
If the boundary marked point $z_0$ and interior marked point $z^+_1$ are
in the same component, then the contribution from the bubble component is $m_0^b$ which is a multiple of unit.
Therefore such contribution vanishes by the usual  argument. The other case that $z_0$ lies in a different disc component than $z_1^+$,
then it can be written as $m_1^b(\ks(R))$. We may take the projection to cohomology. Namely,
we have a quasi-isomorphism from $(\mathcal{B}, \{m_k^b\})$ to $\mathcal{B}_{can}, \{m_{k,can}^b\}$
and denote its linear component (part with one input) by $\prod^b$. We have
$\prod^b \circ m_1^b = m_{1,can}^b \prod^b$ and hence we obtain the result. 
\end{proof}

Now, let us show that $\ks$ is a ring homomorphism. The geometric idea behind this is rather well-known.  In the language of 2-dimensional open-closed topological conformal field theory, the closed-open map is a ring homomorphism. But the actual proof of this in \cite{FOOO_MS} becomes
rather technical. One of the issue there is that $T^n$-equivariance is essential to define $\ks$-map, but the moduli space of $J$-holomorphic spheres does not have $T^n$-equivariant perturbation so they need to find a way to accommodate both. In \cite{ACHL}, a simplified construction without such $T^n$-equivariance is given.  We follow closely the construction in \cite{ACHL}. The main simplification is that
we can use a uniform Kuranishi perturbation scheme, namely,  a component-wise compatible continuous family of multi-section by Fukaya \cite{F} for every moduli spaces involved.

Consider the map 
$$\mathfrak{forget}: \mathcal{M}^{main}_{k+1,2}(\beta, A \otimes B)  \to \mathcal{M}^{main}_{1,2}$$
of forgetting the map and the boundary marked points except the first one.
We are interested in $\mathfrak{forget}^{-1}(\Sigma_0), \mathfrak{forget}^{-1}(\Sigma_{12})$ where
$\Sigma_0$ is  a stable disc with a sphere attached at the interior, and $\Sigma_{12}$ is a disc with two disc bubble each of which
contains an interior marked point (See Figure \ref{disk}).
The fiber product
\begin{equation}\label{eq:0q}
\mathfrak{forget}^{-1}(\Sigma) \;_{(ev_1^+, ev_2^+)} \times_{(M \times M)} (f_1 \times f_2)
\end{equation}
For $\Sigma = \Sigma_0$, the above expression  gives $\ks(f_1 \ast_Q f_2)$, and
for $\Sigma = \Sigma_{12}$, the above expression  gives $(-1)^{\deg \ks(f_1)} m_2(\ks(f_1), \ks(f_2))$.
For an interval $I \subset  \mathcal{M}^{main}_{1,2}$ connecting $\Sigma_0, \Sigma_{12}$, \eqref{eq:0q} with $\Sigma=I$
should give us the cobordism connecting these two operations. The rest of the proof is the same as \cite[Section 6.2]{ACHL} 
and \cite{FOOO_MS} and
we omit the details. But the sign $(-1)^{\deg \ks(f_1)}$ did not appear previously due to even dimensionality of cohomology classes, and we will explain this later in the section.

\subsection{Proof of Theorem \ref{thm:ks2}}
We label one of the embedded lift of $\OL{L}$ as $\WT{L}_1$ and label the rest of the lift by
 $\WT{L}_g = g \cdot (\WT{L}_1)$. Denote by $\WT{L} = \bigoplus_{g\in G} \WT{L}_g$.
 Th idea is to consider the moduli space of $J$-holomorphic stable polygon in $M$ (similar to Figure \ref{disk}) with interior insertion of $A$
 with boundary on $\WT{L}$. We write it as $\mathcal{M}_{k+1,1}(\beta,A)$ for $\beta  \in H_2(M, \WT{L})$. 
From Lemma \eqref{lem:isobr}, we have a $G$-equivariant $\AI$-isomorphism between 
\[ \big( \CB \rtimes \HG, \{m_k^{b\otimes 1}\}\big) \simeq \big(  CF(\WT{L},\WT{L}) \otimes  \Lambda[x_1,\cdots,x_n], \{m_k^{\WT{b}}\} \big).\]
%
Thus we can define $\ks$ following \eqref{eq:defks} as a map to the cohomology of the right hand side.
We can prove that $\ks$ is well-defined in cohomology following Lemma \ref{lem:ksc}.

Now, $G$ acts on the moduli space $\mathcal{M}_{k+1,0}(\beta)$ freely
and  we  consider a $G$-equivariant Kuranishi structure and perturbations. This defines a Fukaya $\AI$-algebra of Lagrangians
$CF(\WT{L},\WT{L})$ with a strict $G$-action on it, and its $G$-invariant part defines an $\AI$-algebra of $\OL{L}$.
Suppose that ambient cycle $A$ is $G$-invariant. 
then moduli space for the Kodaira-Spencer map $\mathcal{M}_{k+1,1}(\beta,A)$ has a $G$-action on it, and the output of $\ks(A)$ become $G$-invariant as well. By $\Phi^{-1}$, it corresponds to an element in  the trivial sector $\mathcal{B} \otimes 1$.
In terms of $\Jac(W,\WH{G})$, for a $G$-invariant cycle $A$, $\ks(A)$ should lie in the trivial sector of $\Jac(W,\WH{G})$
which is $\Jac(W)^{\WH{G}}$.   Since $H^*([M/G],\Lambda) \cong H^*(M,\Lambda)^G$ (we consider $\Lambda$ with
$\C$-coefficient as usual) $A$ may be considered as an element of $H^*([M/G])$. 
Therefore, this moduli space gives two associated maps, one from $QH^*(M)$ and  one from $H^*([M/G])$.

Now, let us consider more interesting case that $A$ is not  $G$-invariant. Roughly $\ks(A)$ will be in non-trivial  sectors of $\Jac(W,\WH{G})$.
From finite abelian $G$-action, we can decompose $H^*(M)$ according to each character $\chi \in \WH{G}$.
Namely if $A$ is in $\chi$-eigenspace, $g \cdot A = \chi(g) A$. By considering $G$-equivariant perturbations of
the moduli space $$\bigcup_{g \in G} \mathcal{M}(g \cdot \beta, g \cdot A),$$ we observe that $\ks(A)$ lies in $\chi$-eigenspace as well.
Namely,  we have
$$ \Phi^{-1} (\ks(A)) \in \mathcal{B} \otimes \chi.$$
Now, we consider the second  $\WH{G}$-action on $\mathcal{B}$ (see Definition \ref{def:hatac}).
\begin{lemma}
We claim that  $$\Phi^{-1} (\ks(A)) \in (\mathcal{B} \otimes \chi )^{\WH{G}}.$$
\end{lemma}
\begin{proof}
Recall that $\WH{G}$-action acts on both Lagrangian Floer generators (by recording difference of branches of $\WT{L}$) and on variables (from its action on the associated immersed sector at the quotient), but not on $\chi$.
Now, $\WH{G}$-action on $\ks(A)$ come from two sources, one is the variables corresponding to $b$'s, and the other is the output Floer generator in $CF(\WT{L},\WT{L})$.
Given a $J$-holomorphic polygon $u$ with $k+1$ boundary marked points,
the boundary is given by $\WT{L}$ contributing to $\ks(A)$. Let us choose $g_0, \cdots, g_k \in G$ such that
$j$-th marked point $z_j$ maps to the Lagrangian intersection $L_{g_i} \cap L_{g_{i+1}}$ for $i=0,\cdots,k \mod k$.
At $j$-th marked point, the next branch of lagrangian is obtained by the multiplication of $g_i^{-1}g_{i+1} \in G$
of the previous branch (in a counter-clockwise order), hence $\chi \in \WH{G}$ action on the
dual variable will give $\chi ( g_i^{-1}g_{i+1})$ for each $i =0,1,\cdots, k-1$. 
Hence $\chi$-action on resulting variables on $\ks(A)$ is 
$$\chi ( g_0^{-1}g_{1}) \cdots \chi ( g_{k-1}^{-1}g_{k}) = \chi (g_0^{-1} g_k).$$
On the other hand, the output of $\ks(A)$ lies in $CF(L_{g_0}, L_{g_k})$ and
by Definition \ref{def:hatac}, $\chi$ action on the output is given by the multiplication of $\chi\big( (g_0^{-1} g_k)^{-1}\big)$.
Therefore, these two actions cancel each other out and $\ks(A)$ is  invariant under this additional $\WH{G}$-action.
\end{proof}

\subsection{Sign computation}
Let us explain the sign in the theorem \ref{thm:ks}.
$$\ks(f_1 \ast_Q f_2) = (-1)^{\deg \ks(f_1)} m_2(\ks(f_1), \ks(f_2))$$
One can check $\deg_L (\ks(f)) = \deg_M (f)$.
In \cite{FOOO_MS}, the above sign did not appear since non-trivial cohomology classes of toric manifolds 
(which are $(\C^*)^n$-invariant) are of even degree.

Note that \eqref{eq:0q} can be identified with (by orientation formulae of \cite[Chapter 8]{FOOO})
$$\big(\mathfrak{forget}^{-1}(\Sigma_0) \;_{ev_1^+} \times_M f_1 \big) \;_{ev_2^+} \times_M f_2$$
 
Let us consider the sign for $\mathfrak{forget}^{-1}(\Sigma_{12})$. We start by recalling the following orientation formula for disk bubbling.
\begin{prop}\cite[Proposition 8.3.3]{FOOO} We have an isomorphism
$$ \partial \mathcal{M}^{reg}_{m+1}(\beta)
= \bigcup_{\beta = \beta'+\beta''} (-1)^{(m_1-1)(m_2-1)+(n+m_1)} \mathcal{M}^{reg}_{m_1+1}(\beta') \,_{ev_1} \times_{ev_0} \mathcal{M}^{reg}_{m_2+1}(\beta'')$$
 as oriented Kuranishi spaces where $m=m_1+m_2-1$ and $n=\dim(L)$.
\end{prop}

By applying the above formula twice, we identify a stratum in  $\partial^2 \mathcal{M}^{reg}_{m+1}(\beta)$
(as oriented Kuranishi spaces) given by
$$ \bigcup_{\beta= \beta_1+\beta_2+\beta_3}
\big( \mathcal{M}^{reg}_{3}(\beta_1) \,_{ev_1}\times_{ev_0} \mathcal{M}^{reg}_{1}(\beta_2) \big) \,_{ev_2} \times_{ev_0} \mathcal{M}^{reg}_{1}(\beta_2)$$
There is an analogous formula with  interior marked points
$$ \bigcup_{\beta= \beta_1+\beta_2+\beta_3}
\big( \mathcal{M}^{reg}_{3}(\beta_1) \,_{ev_1}\times_{ev_0} \mathcal{M}^{reg}_{1,1}(\beta_2) \big) \,_{ev_2} \times_{ev_0} \mathcal{M}^{reg}_{1,1}(\beta_2)$$
We take a fiber product $\big(\bullet \,_{ev^+} \times_{M} f_1 \big) \,_{ev^+}\times_M f_2$ of the above from the right
and after some yoga of moving around fiber products (which we leave as an exercise, but this does not bring any additional sign mainly because $M$ is even dimensional),  we get the following
$$\bigcup_{\beta_1}\big( \mathcal{M}^{reg}_{3}(\beta_1) \,_{ev_1}\times_{ev_0} \ks(f_1) \big) \,_{ev_2} \times_{ev_0} \ks(f_2)$$
Fukaya-Oh-Ohta-Ono defined $m_2$-product to be the above with an additional sign.
$$(-1)^{\deg_L \ks(f_1)}  m_2(  \ks(f_1),  \ks(f_2))$$
Here $\deg_L \ks(f_1) = \deg_M f_1$. Now we can argue as in \cite{FOOO_MS} to find a cobordism
between $\mathfrak{forget}^{-1}(\Sigma_0)$ and $\mathfrak{forget}^{-1}(\Sigma_{12})$.

\section{Matching algebraic and geometric generators for orbifold Jacobian algebra}\label{sec:ag}
We have shown in Theorem \ref{thm:main} under the  Assumption \ref{as2} that we can use Floer theory to construct an $\AI$-algebra whose cohomology algebra is $\Jac(W,\WH{G})$. But the proof thereof does not specify the precise isomorphism. 
In this section, we will describe more precise correspondence and  identify geometric generators corresponding to the formal generators $\xi_h$ of
orbifold Jacobian algebra in Definition \ref{def:ojr}.

For this, let us recall briefly the construction of \cite{Shk}.
For a curved algebra $R_W$, and its bimodule $M$ (which means it is an $R$-bimodule whose left and right action by $W$ agree),
we have a mixed complex $(\mathcal{K}^*(R_W,M))$ whose total cohomology is Hochschild cohomology $H^*(R_W,M)$:
To define this, we set $\mathcal{K}_*(R)=  R^e\otimes k[\theta_1,\cdots,\theta_n]$ which is a mixed complex with
degree 1 differential $\delta_{\mathrm{Kos}}$ and degree $(-1)$ differential $\delta_{\mathrm{curv}}$ given by
$$ \delta_{\mathrm{Kos}}:=\sum_{i=1}^n (x_i-y_i)\partial_{\theta_i}, \delta_{\mathrm{curv}}:=\sum_{i=1}^n \nabla_i^{x \to (x,y)}(W)\cdot \theta_i $$
Also bimodule $M$ may be considered as a $R^e$-module with $x_i$ acting on the left and $y_i$ acting on the right.
Then  mixed complex $(\mathcal{K}^*(R_W,M))$ is defined as
$$\big( hom_{R^e}( \mathcal{K}_{-*}(R), M), \partial_{\mathrm{Kos}} =\delta_{\mathrm{Kos}}^\vee, 
\partial_{\mathrm{curv}}=\delta_{\mathrm{curv}}^\vee \big).$$

This mixed complex can be identified (see (4.10) \cite{Shk}) with the following 
$$\big( M[\partial_{\theta_1},\cdots,\partial_{\theta_n}], \partial_{\mathrm{Kos}}:= \sum_{i=1}^n (x_i-y_i)\partial_{\theta_i},
\partial_{\mathrm{curv}}:=-\sum_{i=1}^n \nabla_i^{x \to (x,y)}(W)\cdot \theta_i \big)$$

In the case of $M = R$ (or $R_W$),  left and right action agree and hence we may set $x_i= y_i$.
Also, note that 
$$ \nabla_i^{x \to (x,y)}(W) \mid_{x=y} = \partial_i W(x).$$
Hence,  $\mathcal{K}^*(R_W,R_W)$ computing Hochschild cohomology $H^*(R_W,R_W)$ equals
$$\big( R[\partial_{\theta_1},\cdots,\partial_{\theta_n}], \partial_{\mathrm{Kos}}=0,
\partial_{\mathrm{curv}}:=- \sum_{i=1}^n \partial_{x_i}W(x)\cdot \theta_i \big)$$
whose cohomology is $\Jac(W)$ if $W$ has isolated singularities.

For $\WH{G}$-equivariant setting, consider the semi-direct product $R_W[\WH{G}]=R_W \otimes k[\WH{G}]$
with product structure $$(r \otimes g) \cdot (r' \otimes g'):=r(g\cdot r') \otimes gg'.$$
It is an $R$-bimodule. On $R_W \otimes \chi$,   $x$  acts on the right by $\chi \cdot x$.
Hence $\chi$-sector of the complex $\mathcal{K}^*(R_W,R_W[\WH{G}])$ can be identified with 
\begin{equation}\label{eq:ek1}
\big( R[\partial_{\theta_1},\cdots,\partial_{\theta_n}], \delta_{\mathrm{Kos}}(\chi):=\sum_{i=1}^n (x_i-\chi \cdot x_i)\partial_{\theta_i}, \delta_{\mathrm{curv}}(\chi):=- \sum_{i=1}^n \nabla_i^{x \to (x,y)}(W) \mid_{y = \chi \cdot x} \cdot \theta_i \big)
\end{equation}
Let us consider the wedge degree of $R[\partial_{\theta_1},\cdots,\partial_{\theta_n}]$ (here $\partial_{\theta_i}$ has degree one).
Shklyarov showed that the total cohomology of \eqref{eq:ek1} is isomorphic to $\Jac(W \mid_{\mathrm{Fix(\chi)}})\cdot \xi_\chi$
and $1 \cdot \xi_\chi$ corresponds to a cohomology generator whose highest wedge degree term is 
$\prod_{i \in I_\chi} \partial_{\theta_i}$. 

More precisely, in \cite{Shk}, Shklyarov finds a quasi-isomorphic complex whose cohomology is generated by the above term only,
and the quasi-isomorphism (to the original complex) is given by $\exp^{tH_{W,\chi}}$. Since $H_{W,\chi}$ action lowers
the wedge degree, we obtain the above claim.

On the other hand, we can compare \eqref{eq:ek1} to the geometric setting via Assumption \ref{as2}.
Given $\big(V \otimes_{\Lambda_0}R^e, m_1^{b(x),b(y)} \big)$  (which is a matrix factorization of $W(y) - W(x)$), by restricting it to $(\chi^{-1}x,x)$, on which we have a chain complex due to \[W(x) - W(\chi^{-1}\cdot x) =0,\] we obtain the following isomorphism of the complex.
\begin{equation}\label{eq:mk2}
\big(V \otimes_{\Lambda_0}R^e, m_1^{b(x),b(y)} \big)\mid_{(\chi^{-1}x,x)} \simeq \;\; \big( R[ \partial_{\theta_1},\cdots, \partial_{\theta_n}], \sum_i (x_i-\chi^{-1} \cdot x_i) \partial_{\theta_i} + \sum_i \nabla_i^{x \to (x,y)}W \mid_{(\chi^{-1}x,x)} \cdot \theta_i\big).
\end{equation}
By change of variables $\chi^{-1} x_i \to x_i$, we have an isomorphism
\[\big(V \otimes_{\Lambda_0}R , m_1^{b(\chi^{-1} \cdot x),b(x)} \big) \simeq
\big( R[ \partial_{\theta_1},\cdots, \partial_{\theta_n}],  \sum_i (\chi \cdot x_i-x_i) \partial_{\theta_i} + \sum_i \nabla_i^{x \to (x,y)}W \mid_{y=\chi\cdot x} \cdot \theta_i\big).\]
From these observations, we obtain the following.

\begin{prop}\label{cohgenerator}
Given assumption \ref{as2},
 $H^*(\mathcal{B} \otimes \chi, m_1^{b\otimes 1})$ is isomorphic as a module to $\chi$-sector of Hochschild cohomology
 $H^*(R_W, R_W[\WH{G}])$. In this isomorphism, $\xi_\chi$ corresponds to a cohomology generator in $(\mathcal{B} \otimes \chi, m_1^{b\otimes 1} )$  whose highest wedge degree term  (in the isomorphism \eqref{eq:mk2}) is $\prod_{i \in I_\chi} (-\partial_{\theta_i})$.
 \end{prop}
 \begin{remark}
 The sign $(-1)^{|I_\chi|}$ in the identification also appeared in the construction of \eqref{comparingmap}. See Lemma \ref{signchainmap}.
 \end{remark}
 
\section{Orbifold Jacobian algebra for $T^2$}\label{sec:t2}
Take $T^2$ which is obtained by identifying opposite sides of a regular hexagon. There is a $\Z/3$-action given by rotation, and
the quotient is an orbifold sphere $\mathbb{P}^1_{3,3,3}$. We can consider an immersed Lagrangian $\bL \subset \mathbb{P}^1_{3,3,3}$
equipped with a non-trivial spin structure (for the symmetry, we follow \cite{CHLnc} to equip $\bL$ with a unitary line bundle whose
holonomy is $(-1)$ which is equally distributed). This is the Seidel Lagrangian \cite{S} (see also Efimov \cite{E}).  $\bL$ lifts to embedded circles in $T^2$, which are denoted by $L_1,L_g,L_{g^2}$ where $\{1,g,g^2\}$ denote elements of $\Z/3$. See Figure \ref{fig:2torus}.

\begin{figure}
\includegraphics[height=2.5in]{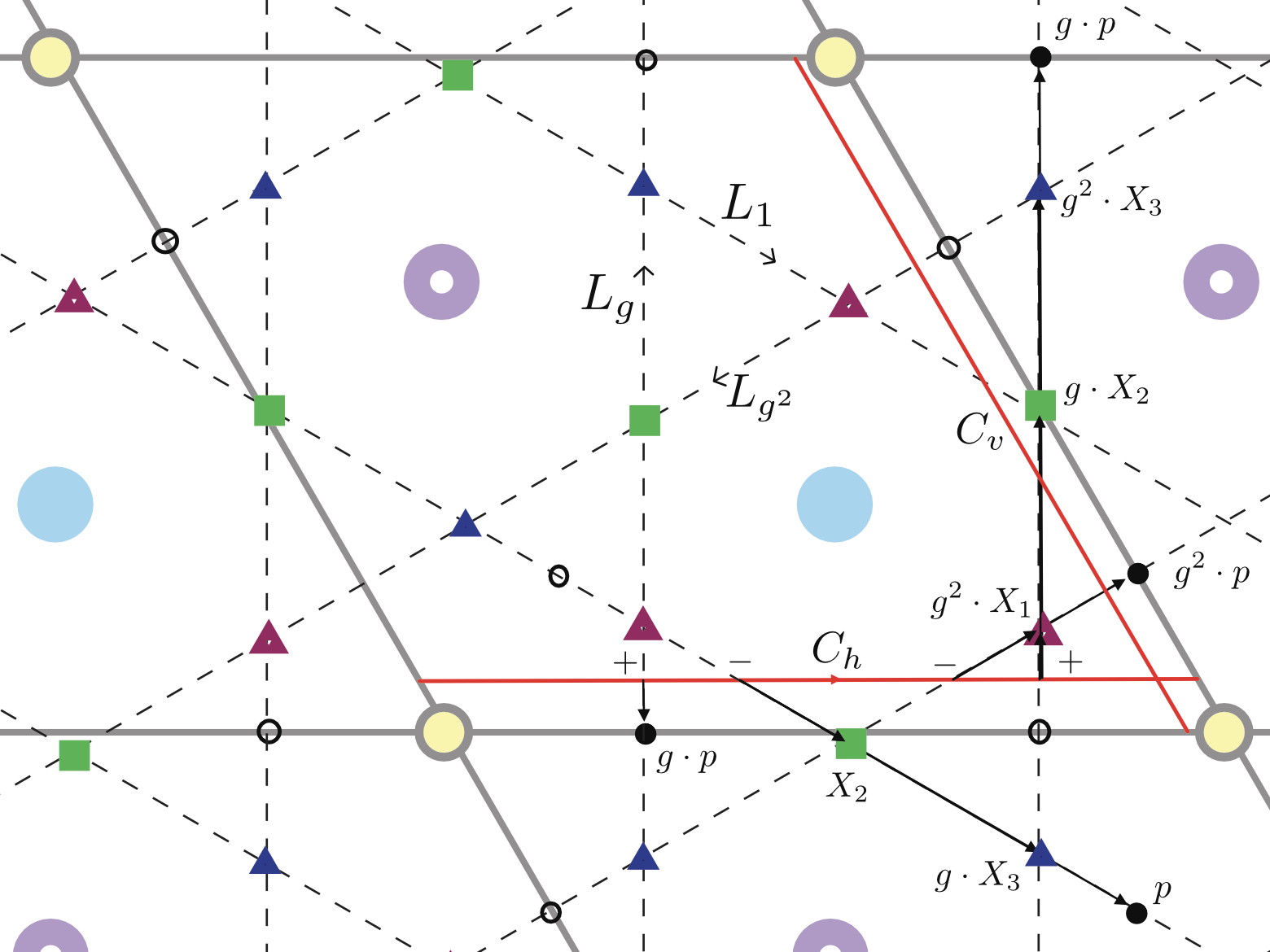}
\caption{Dotted lines are embedded Lagrangians $L_1$, $L_g$ and $L_{g^2}$. Arrows indicate how to move a point
to one of the critical points of the Morse functions.}
\label{fig:2torus}
\end{figure}

Define
$ q:=T^{\omega(\Delta)}$, where $\omega(\Delta)$ is the area of minimal triangle. In \cite{CHL},\cite{CHLnc}, $\bL$ is shown to be weakly unobstructed and the potential $W$ of $\bL$ is computed as follows.
\[W=-\phi(q) i(x_1^3+x_2^3+x_3^3)+\psi(q) i x_1x_2x_3\]
where $\phi$ and $\psi$ are elements of Novikov ring as follows
\begin{equation}\label{phipsi}
 \phi(q):= \sum_{k\in \mathbb{Z}}(-1)^{k+1}(k+\frac{1}{2})q^{(6k+3)^2}, \;\; \psi(q):=\sum_{k\in \mathbb{Z}}(-1)^{k+1}(6k+1)q^{(6k+1)^2}.\end{equation}
Dual group $\WH{G}$-action are given as follows. Define \[ \cchi := e^{2\pi i/3},\]
and denote elements of $\WH{G}$ by  $\{ 1,\chi,\chi^2\}$ where $\chi \cdot x_i= \cchi x_i$. We remark that the mirror potential given in \cite{E} and \cite{S} is just $W_{alg}=x_1^n+x_2^n+x_3^n + x_1x_2x_3$ whereas the disc potentials in \cite{CHL} and \cite{CHKL} are quantum corrected ones and they are infinite series for hyperbolic cases. Orbifold Jacobian algebra
structure for $\Jac(W_{alg},\Z/n)$ has been computed in the appendix of \cite{Shk}.

We are interested in the Kodaira-Spencer map from $QH^*(T^2)$ to $\Jac(W,\Z/3)$ (including the quantum corrections). Define an $\AI$-algebra 
\[ \CB:= CF\big((\bL,b(x)),(\bL,b(x))\big)\otimes \Lambda[x_1,x_2,x_3].\]
Since the Seidel Lagrangian for $\mathbb{P}^1_{3,3,3} = [T^2 / \Z/3]$ satisfies Assumption \ref{as2}, we have shown the following
sequence of isomorphisms (except the first map).
\begin{equation}\label{eq:kst2}
QH^*(T^2) \stackrel{\mathfrak{ks}}{\to} H^*\big((\mathcal{B} \rtimes \Z/3)^{\Z/3}\big)_{alg} \cong H^*(MF_{\Z/3}(W), MF_{\Z/3}(W)) \cong \Jac(W,\Z/3).
\end{equation}
We will show that $\mathfrak{ks}$ is an isomorphism and furthermore we will find a precise correspondence as follows.
First, let us introduce some notations for $QH^*(T^2)$ (which equals $H^*(T^2)$ since there are no non-trivial holomorphic spheres).
Let $\be_{T^2}$ be the  fundamental class and $\mathrm{pt}_{T^2}$ be the Poincar\'e dual of the point class in $QH^*(T^2)$.
In addition, we denote Poincar\'e duals of the meridian $C_h$ and
longitude $C_v$ by $[C_h], [C_v]$ respectively. We will take them as in the above Figure \ref{fig:2torus}. Note that
$g \cdot [C_h] = [C_v], g \cdot [C_v]= -[C_h]-[C_v]$.
 Therefore,  $[C_h] - \cchi [C_v], [C_v] - \cchi[C_h]$ are cohomology classes in $\chi$ and $\chi^2$-sector of $H^*(T^2)$ respectively.
\begin{theorem}\label{KS} 
Let the composition map be called $\ks$, by abuse of notation. Then
$$\ks: QH^*(T^2) \to \Jac(W,\Z/3)$$ 
 in \eqref{eq:kst2} is given by 
\begin{eqnarray*}
\be_{T^2} & \mapsto & 1\\
\mathrm{pt}_{T^2} & \mapsto & \frac{1}{24}q\frac{\partial W_\bL}{\partial q} \\
l_\chi:=-i{\gamma} \cdot \frac{[C_h]-\chicheck [C_v]}{\chicheck-1} & \mapsto & \xi_\chi\\
l_{\chi^2}:=-i{\gamma} \cdot \frac{[C_v]-\chicheck [C_h]}{\chicheck-1} & \mapsto & \xi_{\chi^2}.\\
\end{eqnarray*}
Here, $\gamma$ is a modular form 
\begin{equation}\label{eq:g}
 \gamma=\sum_{k \in \Z} (-1)^{k+1} i q^{(6k+1)^2}.
 \end{equation}
\end{theorem}
\begin{proof}

Let us first compute the products in $\Jac(W,\Z/3)$.
\begin{lemma}
In $\Jac(W,\Z/3)$, we have the following product 
$$\xi_\chi \bullet \xi_{\chi^2} = \sigma_{\chi,\chi^2} \cdot 1$$
where $\sigma_{\chi,\chi^2}$ can be computed (following the definition in \cite{Shk}) to be
\[ \sigma_{\chi,\chi^2}=\frac{(27\phi^3-\psi^3)i}{(1-\chicheck)^3}x_1x_2x_3 \in \Jac(W).\]
We have $\xi_\chi \bullet \xi_\chi =0$, $\xi_{\chi^2} \bullet \xi_{\chi^2} =0$ from the degree condition \eqref{degcond}.
\end{lemma}
We will prove this lemma in Appendix \ref{algjac}.

On the other hand, one can easily check that $l_\chi \cup l_\chi =0, l_{\chi^2} \cup l_{\chi^2} =0$.
Also, it is easy to compute that
\begin{equation}
l_\chi \cup l_{\chi^2} = -\gamma^2   \cdot \left( \frac{1+\chicheck}{1-\chicheck} \right) \cdot \mathrm{pt}_{T^2}.
\end{equation}
Now, we are ready to check that $\ks$ is a ring isomorphism. Observe that the Jacobian relation identifies monomials $x_1x_2x_3, x_1^3,x_2^3,x_3^3$ up to scaling and we can use it to show that $$\frac{1}{24}q\frac{\partial W}{\partial q} = 
i \frac{q}{24}\big(-\frac{\psi}{\phi}\frac{\partial\phi}{\partial q}+ \frac{\partial\psi}{\partial q}\big)x_1x_2x_3 \in \Jac(W).$$

Also, we need the following identity of modular forms, which is proved in \cite{CLS}.
\begin{equation}\label{eq:modular}
\gamma^2 q\Big(-\psi \frac{\partial \phi}{\partial q}+\phi \frac{\partial \psi}{\partial q}\Big)=-8\phi(27\phi^3-\psi^3).
\end{equation}
By this identity, we can show the most nontrivial part
$$\ks(l_\chi \cup l_{\chi^2}) = \ks(l_\chi) \bullet \ks(l_{\chi^2})$$
as follows:
 \begin{align*}
\ks(l_\chi \cup l_{\chi^2}) &= -\gamma^2   \cdot \left( \frac{1+\chicheck}{1-\chicheck} \right)  \ks(\mathrm{pt}_{T^2}) \\
&=  -\gamma^2   \cdot \left( \frac{1+\chicheck}{1-\chicheck} \right) i \frac{q}{24}\big(-\frac{\psi}{\phi}\frac{\partial\phi}{\partial q}+\frac{\partial\psi}{\partial q}\big)x_1x_2x_3 \\
&=\left( \frac{1+\chicheck}{1-\chicheck} \right) i  \frac{8(27\phi^3-\psi^3)}{24}x_1x_2x_3 \\
&=\frac{i(27\phi^3-\psi^3)}{(1-\chicheck)^3}x_1x_2x_3 = \xi_\chi \bullet \xi_{\chi^2}. 
\end{align*}
It finishes the proof of Theorem \ref{KS}.
\end{proof}

Now, let us explain each map in \eqref{eq:kst2} to justify our choice of $l_\chi$ and $l_{\chi^2}$.
 
 First we will investigate images of $\ks$ in $H^*((\mathcal{B} \rtimes \Z/3)^{\Z/3})$. Recall that $m_1^b$ preserves eigenspaces of $\WH{G}$-action.
Let us first consider the trivial sector $1 \in \WH{G}$ whose $m_1^b$ cohomology is given by
$$H^*\big((\mathcal{B} \otimes 1)^{\Z/3}, m_1^b\big) \cong H^*(\mathcal{B}, m_1^b)^{\Z/3} \cong \Jac(W)^{\Z/3}$$
Note that $\Jac(W)$ is generated by 8 elements $1,x_1,x_1^2,x_2,x_2^2,x_3,x_3^2,x_1 x_2 x_3$ and $\Z/3$ acts on
variables by multiplication of $3$rd root of unity.
 Hence  $\Jac(W)^{\Z/3}$ is generated by $1$ and $x_1 x_2 x_3.$

In fact, we can choose a specific generator using the $\ks$ map.
Since $\ks$ is a ring homomorphism, it sends the unit of quantum cohomology to the unit $\be_L \otimes 1$.
The class of $x_1 x_2 x_3$ comes from $\ks$ map of the point class of $T^2$. 
In \cite{ACHL}, they showed that $\ks$ map is an isomorphism for $\mathbb{P}^1_{3,3,3}$ and computed $\ks(\mathrm{pt}_{\PP^1_{3,3,3}})$. The same computation applies to our case and we have
\[ \ks(\mathrm{pt}_{T^2})=\frac{1}{24}q\frac{\partial W}{\partial q}\cdot (\be_L\otimes 1).\]
To prove this, one divides a point class into 1/24 point classes each of which lies in different pieces of $T^2$ separated by the Seidel Lagrangian. Since each piece is set to have equal area $q$, count of disks with a point insertion can be compared to the area of the corresponding disc. This gives the above result. 

To find cohomology classes in $\chi, \chi^2$-sectors of $QH^*(T^2)$ that matches with $\xi_\chi, \xi_{\chi^2}$,
we start with Kodaira-Spencer images of two circles $[C_h], [C_v]$ generating $QH^1(T^2)$.
Note that $C_h$ intersect $\bL$ at 4 points
and hence $\mathfrak{ks}([C_h])$ is given by Poincare duals of these 4 points(see Figure \ref{fig:2torus}). 
We wish to compare these points, but one should note that  two different points in $\bL$ are not $m_1^{b,b}$
cohomologous in general. Let us argue in terms of singular chains(which can be made into that of differential forms). For a line segment $I$ connecting two points $p,p'$ in $\bL$ (in the domain of immersion
and oriented as $\bL$),
$m_1^{b,b}(I)$ should have the classical term $\partial I = p' - p$ as well as the constant disc contributions coming from
$m_2(b,I), m_2(I,b)$. Since $b$ comes from immersed sector, only one of these two are composable, and the composable
ones can  be computed following the case of the unit identity 
\[m_2(b,1_\bL) = -b,\;\; m_2(1_\bL,b) =b.\]
Using this idea, we compare each of 4 intersection points of $C_h \cap \bL, C_v \cap \bL$ with chosen minimum points $p,\; g\cdot p,\; g^2\cdot p$
of the Morse function on $L_1,L_g,L_{g^2}$ respectively (see Figure \ref{fig:2torus}). 
From the Figure \ref{fig:2torus}, it is not hard to obtain the following.
\begin{prop}
We have
\[ \ks([C_h])=2g\cdot p-p-g^2\cdot p+x_2(g^2\cdot X_{2}-X_{2})+x_3(g\cdot X_{3}-X_{3}),\]
\[ \ks([C_v])=2g^2\cdot p-p-g\cdot p+ x_2(X_2-g\cdot X_{2})+x_3(g^2\cdot X_{3}-g\cdot X_{3}). \]
\end{prop}
Then, by Lemma \ref{lem:qq}, we have
\begin{align*}
 \Phi_1^{-1}\big(\ks([C_h])\big)&=\chicheck p\otimes\chi+\chicheck^2 p\otimes\chi^2\\
 &+\frac{\big((\chicheck^2-1)x_2X_2+(\chicheck-1)x_3X_3\big)\otimes\chi+ \big((\chicheck-1)x_2X_2+ (\chicheck^2-1)x_3X_3\big)\otimes\chi^2}{3}, 
 \end{align*}
\begin{align*}
 \Phi_1^{-1}\big(\ks([C_v])\big)&=\chicheck^2 p\otimes\chi+ \chicheck p\otimes\chi^2\\
 &+\frac{\big((1-\chicheck)x_2X_2+(\chicheck^2-\chicheck)x_3X_3\big)\otimes\chi+\big((1-\chicheck^2)x_2X_2+(\chicheck-\chicheck^2)x_3X_3\big)\otimes\chi^2}{3}. 
 \end{align*}
Hence, we have
\begin{equation}\label{chicocycle} \Phi_1^{-1}\Big(\ks\big(\frac{[C_h]-\chicheck[C_v]}{\chicheck-1}\big)\Big)= \big(p +\frac{\chicheck^2}{\chicheck-1}x_2X_2-\frac{1}{\chicheck-1}x_3X_3\big)\otimes\chi
\end{equation}
\begin{equation}\label{chi2cocycle} \Phi_1^{-1}\Big(\ks\big(\frac{[C_v]-\chicheck[C_h]}{\chicheck-1}\big)\Big)= \big(p -\frac{\chicheck^2}{\chicheck-1}x_2X_2+\frac{\chicheck}{\chicheck-1}x_3X_3\big)\otimes\chi^2.
\end{equation}
Note that $([C_h]-\chicheck[C_v])$ and $([C_v]-\chicheck[C_h])$ are in fact $\chi$ and $\chi^2$-eigenvectors for $\Z/3$-action on $H^1(T^2)$, and the above illustrates that $\ks$ preserves eigenspaces.

Let us check that \eqref{chicocycle} and \eqref{chi2cocycle} are indeed nontrivial cocycles in $\mathcal{B} \rtimes \Z/3$.
From Lemma \ref{lem:twisteddiff}, we may consider the chain complex $\big(hom(\OL{L},\OL{L})\otimes R,m_1^{b(\chi^{-1}),b} \big)$. 

%

By \eqref{mfabc} with substitution $(x,y) \mapsto (\chi^{-1}\cdot x,y)$, we have
\begin{align*}
m_1^{b(\chi^{-1}),b}(e)&=\sum_{1\leq i \leq 3}(1-\check{\chi}^2)x_i X_i,\\
m_1^{b(\chi^{-1}),b}(\bar{X}_1^{new})&= (1-\check{\chi}^2)x_1 p^{new}-c_{12}X_2-c_{13}X_3,\\
m_1^{b(\chi^{-1}),b}(\bar{X}_2^{new})&= (1-\check{\chi}^2)x_2 p^{new}+c_{12}X_1-c_{23}X_3,\\
m_1^{b(\chi^{-1}),b}(\bar{X}_3^{new})&= (1-\check{\chi}^2)x_3 p^{new}+c_{13}X_1+c_{23}X_2.
\end{align*}
Thus, for any linear combination of $\{e,\bar{X}_1^{new},\bar{X}_2^{new},\bar{X}_3^{new}\}$, its image of $m_1^{b(\chi^{-1}),b}$ cannot be $p+\frac{\check{\chi}^2}{\check{\chi}-1}x_2 X_2-\frac{1}{\check{\chi}-1} x_3 X_3$, because the coefficient of $p$ of the image is always a polynomial with zero constant term. In the same way, we can also prove that $p-\frac{\check{\chi}^2}{\check{\chi}-1}x_2 X_2+\frac{\check{\chi}}{\check{\chi}-1} x_3 X_3$ is not an image of $m_1^{b(\chi^{-2}),b}$. This proves that \eqref{chicocycle} and \eqref{chi2cocycle} are nontrivial cocycles.
%
%
%

Since $\Jac(W,\HG)_\chi$ and $\Jac(W,\HG)_{\chi^2}$ are 1-dimensional in this case, and same is true for $\chi$ and $\chi^2$-eigenspaces of $H^1(T^2)$. Therefore, we find that Kodaira-Spencer map sends $\chi$ and $\chi^2$-sector of $QH^*(T^2)$ to the following spaces as a bijection.
\[H^*(\mathcal{B}\otimes \chi,m_1^{b\otimes 1}) \cong \Jac(W,\HG)_\chi,\;\;H^*(\mathcal{B}\otimes \chi^2,m_1^{b\otimes 1}) \cong \Jac(W,\HG)_{\chi^2}.\] 

To find an exact generator that matches with $\xi_\chi$ and $\xi_{\chi^2}$, we use Proposition \ref{cohgenerator}.
Namely, we look for cocycles in $\big(hom(\OL{L},\OL{L}),m_1^{b(\chi^{-1}),b} \big)$ and $\big(hom(\OL{L},\OL{L}),m_1^{b(\chi^{-2}),b} \big)$ respectively that has the highest wedge degree terms  given by $-p^{new}=-\gamma p$.
 (The sign appears because $ |I_{\chi}|=|I_{\chi^2}|=3$.)
 Therefore, we modify  \eqref{chicocycle} and \eqref{chi2cocycle} to obtain
\[ -\gamma\cdot \frac{[C_h]-\chicheck [C_v]}{\chicheck-1},\; -\gamma\cdot\frac{[C_v]-\chicheck [C_h]}{\chicheck-1}  \in H^*(T^2)\]
respectively. It concludes the explanation of the choice of $l_\chi$ and $l_{\chi^2}$ in Theorem \ref{KS}.
\begin{remark}
We can remove $i=\sqrt{-1}$ in the statement of the theorem if we consider $H^*(T^2)_{alg}$ (or alternatively $\Jac(W,\Z/3)_{op}$).
\end{remark}

\appendix
\section{Algebraic computation of orbifold Jacobian product}\label{algjac}
In \cite[Appendix~A]{Shk}, Shklyarov computed orbifold Jacobian algebra structures for 
\[ W:=x_1^{2g+1}+x_2^{2g+1}+x_3^{2g+1}-x_1 x_2 x_3,\]
with the action by 
\[ G:=\{ (\zeta,\zeta,\zeta^{-2})\in (\mathbb{C}^*)^3 \mid \zeta^{2g+1}=1\}.\]
Following his computation, we compute products of $\Jac(W,\Z/3)$
for the quantum corrected potential
\[W=-\phi(q) i(x_1^3+x_2^3+x_3^3)+\psi(q) i x_1x_2x_3\]
where $\Z/3=\{1,\chi,\chi^2\}$ acts on $\Lambda[x_1,x_2,x_3]$ by $ \chi\cdot x_i=e^{2\pi i/3} x_i$.  Hence $\chicheck:=e^{2\pi i/3}$.

To compute $\sigma_{\chi,\chi'}$,  we first compute
$H_{W}(x,\chi\cdot x,x)$ and $H_{W,\chi}(x)$ for $\chi\in \widehat{G}$. Recall that
\[ H_{W}(x,y,z)=\sum_{1\leq i\leq j \leq 3} \big(\nabla_i^{y \to (y,{z})}\nabla_j^{x\to (x,y)}W\big)\theta_j \otimes \theta_i.\]

The constant $\sigma_{1,\chi}$ is easily computed as follows. One can check that $d_{1,\chi}=0$.
\begin{align*} 
\sigma_{1,\chi}&= \frac{1}{d_{1,\chi}!} \Upsilon\big( (\lfloor H_W(x,\chi\cdot x,x)\rfloor_{\chi}+\lfloor H_{W,1}(x)\rfloor_{\chi}\otimes 1 + 1\otimes \lfloor H_{W,\chi}(\chi\cdot x)\rfloor_{\chi})^{d_{1,\chi}}\otimes \partial_{\theta_{I_1}} \otimes \partial_{\theta_{I_\chi}}\big) \\
&= \Upsilon\big( (1\otimes 1) \otimes (1 \otimes \partial_{\theta_{I_\chi}})\big)=\Upsilon(1 \cdot \partial_{\theta_{I_\chi}})=1
\end{align*}
where the last identity is due to the definition of $\Upsilon$. We conclude that $\xi_1$ is the multiplicative identity of $\Jac'(W,\Z/3)$.

Now we compute $\sigma_{\chi,\chi'}$ for both $\chi$ and $\chi'$ are not $1$. We have $\sigma_{\chi,\chi}=0$ for $\chi\neq 1$, because
\[ d_{\chi,\chi}=\frac{d_\chi+d_\chi-d_{\chi^2}}{2}=\frac{3}{2}\notin \Z \]
Let us compute  $\sigma_{\chi,\chi^2}$. From the definition of differences, we have (for $j=1,2,3$)
\[ \big(\nabla_i^{y \to (y,z)}\nabla_i^{x \to (x,y)}W\big)(x,\chi\cdot x,x)\theta_j\otimes \theta_j=- \frac{3\phi i x_j}{1-\chicheck}\theta_j \otimes \theta_j, \]
\[ \sum_{1\leq i<j\leq 3}\big(\nabla_i^{y \to (y,z)}\nabla_j^{x \to (x,y)}W\big)(x,\chi\cdot x,x)\theta_j\otimes\theta_i= \psi i x_3 \theta_2\otimes \theta_1+\chicheck \psi ix_2 \theta_3 \otimes \theta_1+\psi ix_1 \theta_3\otimes \theta_2.\]
So we have 
\begin{align}
\begin{split}
 H_{W}(x,\chi\cdot x,x)=&-\frac{3\phi ix_1}{1-\chicheck}\theta_1 \otimes \theta_1-\frac{3\phi ix_2}{1-\chicheck}\theta_2 \otimes \theta_2-\frac{3\phi ix_3}{1-\chicheck}\theta_3 \otimes \theta_3\label{eq:Hw}\\
 &+\psi i x_3 \theta_2\otimes \theta_1+\chicheck \psi ix_2 \theta_3 \otimes \theta_1+\psi ix_1 \theta_3\otimes \theta_2.
 \end{split}
 \end{align}
We also compute
\begin{equation}\label{eq:Hwchi}
 H_{W,\chi}(x)=-\frac{\chicheck\psi ix_3}{1-\chicheck}\theta_1\theta_2-\frac{\chicheck^2\psi ix_2}{1-\chicheck}\theta_1\theta_3,
 \end{equation}
\begin{equation}\label{eq:Hwchi2}
 H_{W,\chi^2}(\chi\cdot x)=-\frac{\psi ix_3}{1-\chicheck^2}\theta_1\theta_2-\frac{\chicheck^2\psi ix_2}{1-\chicheck^2}\theta_1\theta_3,
 \end{equation}
Then $\sigma_{\chi,\chi^2}$ is the constant coefficient of the expression
\[ \frac{1}{3!} \Upsilon\big( (\lfloor H_{W}(x,\chi\cdot x,x)\rfloor_{1}+\lfloor H_{W,\chi}(x)\rfloor_{1}\otimes 1 + 1\otimes \lfloor H_{W,\chi^2}(\chi\cdot x)\rfloor_{1})^{3}\otimes \partial_{\theta_1}\partial_{\theta_2}\partial_{\theta_3} \otimes \partial_{\theta_1}\partial_{\theta_2}\partial_{\theta_3} \big).\]
If we denote coefficients of \eqref{eq:Hw}, \eqref{eq:Hwchi} and \eqref{eq:Hwchi2} by
\[ H_{W}(x,\chi\cdot x,x)=A_{11}\theta_1 \otimes \theta_1+A_{22}\theta_2 \otimes \theta_2+A_{33}\theta_3 \otimes \theta_3+A_{21} \theta_2\otimes \theta_1+A_{31} \theta_3 \otimes \theta_1+A_{32} \theta_3\otimes \theta_2\]
and
\[ H_{W,\chi}(x)=B_{12}\theta_1\theta_2+B_{13}\theta_1\theta_3,\; 
H_{W,\chi^2}(\chi\cdot x)=C_{12}\theta_1\theta_2+C_{13}\theta_1\theta_3,\]
then by definition of $\Upsilon$, we have
\[ \sigma_{\chi,\chi^2}=\lfloor A_{11}A_{22}A_{33}-A_{22}B_{13}C_{13}-A_{33}B_{12}C_{12}+A_{32}B_{12}C_{13}\rfloor_1.\]
Up to the relation
\[3\phi ix_1^3=3\phi ix_2^3=3\phi ix_3^3=\psi ix_1x_2x_3\]
given by Jacobian ideal of $W$, we can check that
\[ \sigma_{\chi,\chi^2}=\frac{(27\phi^3-\psi^3)i}{(1-\chicheck)^3}x_1x_2x_3 \in \Jac(W).\]
Hence, we obtain
\[ \xi_\chi \bullet \xi_{\chi^2}=\frac{(27\phi^3-\psi^3)i}{(1-\chicheck)^3}x_1x_2x_3.\]

\bibliographystyle{amsalpha}

\end{document}